\documentclass[english]{amsart}

\usepackage{amsfonts, amssymb, amsbsy, amsmath, mathrsfs, amsthm,enumerate, bm}

\usepackage{graphicx}
\usepackage{url}
\usepackage{xcolor}
\usepackage{hyperref}
\usepackage{cleveref}

\newcommand\dmu{\, d \mu}
\newcommand\dt{\, d t}
\newcommand\dnu{\, d \nu}
\newcommand\ds{\, d s}

\newcommand\R{\mathbb R}

\usepackage{mathrsfs}

\theoremstyle{definition}
\newtheorem{definition}{Definition}[section]

\newtheorem{remark}[definition]{Remark} 
\theoremstyle{plain}
\newtheorem{theorem}[definition]{Theorem}
\newtheorem{lemma}[definition]{Lemma}
\newtheorem{corollary}[definition]{Corollary}
\newtheorem{proposition}[definition]{Proposition}

\numberwithin{equation}{section}

\crefname{theorem}{Theorem}{Theorems}
\Crefname{theorem}{Theorem}{Theorems}
\crefname{lemma}{Lemma}{Lemmas}
\Crefname{lemma}{Lemma}{Lemmas}
\crefname{proposition}{Proposition}{Propositions}
\Crefname{proposition}{Proposition}{Propositions}
\crefname{corollary}{Corollary}{Corollaries}
\Crefname{corollary}{Corollary}{Corollaries}
\crefname{definition}{Definition}{Definitions}
\Crefname{definition}{Definition}{Definitions}
\crefname{remark}{Remark}{Remarks}
\Crefname{remark}{Remark}{Remarks}
\crefname{example}{Example}{Examples}
\Crefname{example}{Example}{Examples}

\def\Xint#1{\mathchoice
   {\XXint\displaystyle\textstyle{#1}}%
   {\XXint\textstyle\scriptstyle{#1}}%
   {\XXint\scriptstyle\scriptscriptstyle{#1}}%
   {\XXint\scriptscriptstyle\scriptscriptstyle{#1}}%
   \!\int}
\def\XXint#1#2#3{{\setbox0=\hbox{$#1{#2#3}{\int}$}
     \vcenter{\hbox{$#2#3$}}\kern-.5\wd0}}

\def\dashint{\Xint-}

\def\YYint#1#2#3{{\setbox0=\hbox{$#1{#2#3}{\iint}$}
    \vcenter{\hbox{$#2#3$}}\kern-.51\wd0}}

\renewcommand{\div}{\operatorname{div}}

\DeclareMathOperator*{\esssup}{ess\,sup\,}
\DeclareMathOperator*{\essinf}{ess\,inf\,}
\DeclareMathOperator*{\essosc}{ess\,osc\,}

\DeclareMathOperator{\diam}{diam}

\DeclareMathOperator\supp{supp}

\title[Doubling, Poincar\'e and parabolic Harnack inequalities]{Doubling measures, Poincar\'e inequalities and parabolic Harnack inequalities for a doubly nonlinear equation}

\author{Theo Elenius}
\address[Theo Elenius]{Department of Mathematics, Aalto University, P.O. BOX 11100, 00076 Aalto, Finland}
\email{theo.elenius@aalto.fi}
\author{Juha Kinnunen}
\address[Juha Kinnunen]{Department of Mathematics, Aalto University, P.O. BOX 11100, 00076 Aalto, Finland}
\email{juha.k.kinnunen@aalto.fi}

\begin{document}
\bibliographystyle{plainnat}

\begin{abstract}
We characterize metric measure spaces satisfying parabolic Harnack inequalities for 
a doubly nonlinear equation in terms of volume doubling and Poincar\'{e} inequalities. Our approach uses purely analytical methods, based on obtaining estimates for solutions to a related Cauchy problem. This extends previous linear results to the nonlinear setting without relying on heat kernel estimates and representation formulae.
\end{abstract}


\maketitle

\section{Introduction}

A fundamental problem in analysis on metric spaces is to determine whether properties of the underlying space can be characterized via the behavior of solutions to partial differential equations. A well-known result of Grigor'yan \cite{Grigoryan1991} and Saloff-Coste \cite{Saloff-Coste1992, Saloff-Coste2002} states that the parabolic Harnack inequality for the heat equation characterizes volume doubling and a $(1,2)$-Poincar\'e inequality on Riemannian manifolds. The corresponding result has later been obtained in several settings, including locally finite graphs \cite{Delmotte}, Dirichlet forms on locally compact metric spaces \cite{Sturm} and non-local Dirichlet forms \cite{ChenKumagaiWang}, among others.

In this work, we show that the parabolic Harnack inequality for the doubly nonlinear equation 
\begin{equation}
    \label{eq: PDE}
    \partial_t( \lvert u \rvert^{p-2} u)-\div(\lvert\nabla u \rvert^{p-2}\nabla u)=0, \quad 1<p<\infty, 
\end{equation}
characterizes metric measure spaces with a doubling measure and a $(1,p)$-Poincar\'e{} inequality. In particular, this extends the previous results with $p=2$ to the entire parameter range $1<p<\infty$. To the best of our knowledge, this is the first characterization of parabolic Harnack inequalities for a nonlinear parabolic partial differential equation by geometric and functional analytic properties of the underlying space.

The equation \eqref{eq: PDE} was studied by Trudinger \cite{Trudinger}, who observed that nonnegative solutions enjoy a scale and location invariant parabolic Harnack inequality, analogous to the one for the heat equation. The parabolic Harnack inequality asserts that if $u=u(x,t)$ is a nonnegative solution of \eqref{eq: PDE} in $Q=B(x_0,4r)\times (t_0-(4r)^p,t_0+(4r)^p)$, then
\begin{equation}
    \label{eq: PHI}
    \esssup_{Q^-} u \leq c_H \essinf_{Q^+} u,
\end{equation}
where $Q^-=B(x_0,r)\times (t_0- r^p,t_0-2^{-p}r^p)$
and $Q^+=B(x_0,r)\times (t_0+2^{-p}r^p,t_0+r^p)$. Trudinger's parabolic Harnack inequality has been extended to the setting of metric measure spaces with a doubling measure and a $(1,p)$-Poincar\'e{} inequality \cite{KinnunenKuusi2007,MarolaMasson}.
Note, that for $p=2$, the doubly nonlinear equation \eqref{eq: PDE} reduces to the heat equation.

Doubling measures and Poincar\'e inequalities play a central role in analysis on metric measure spaces. Typically these conditions are imposed a priori. For example, they imply Sobolev embedding theorems which are essential in many applications. We refer to the monographs \cite{BjornBjorn2011,Heinonen,HeinonenKoskelaShanmugalingam} for comprehensive accounts of this theory. Semmes \cite{Semmes} gave a geometric characterization of Poincar\'e inequalities in terms of families of curves, but in general few characterizations of these conditions are known.
Our main result establishes a connection between analysis on metric spaces and the regularity theory of nonlinear parabolic partial differential equations. Although our primary motivation comes from the setting of Riemannian manifolds, we work in the more general framework of metric measure spaces.
\begin{theorem}
\label{thm: main thm}
    Let $1<p<\infty$. Let $(X,d,\mu)$ be a metric measure space such that $X$ is geodesic and complete, the measure $\mu$ is finite and nonzero on balls, the Sobolev space $N^{1,p}(X)$ is reflexive and for $p\neq 2$ assume that the Rellich--Kondrachov property in \Cref{def: rellich kondrachov property} holds. Then the following conditions are equivalent:
    \begin{itemize}
        \item[(1)] $\mu$ is doubling and a $(1,p)$-Poincar\'e inequality holds in $(X,d,\mu)$, see \eqref{eq: doubling condition} and \Cref{def: poincare}.
        \item[(2)] The parabolic Harnack inequality holds in $(X,d,\mu)$, see \Cref{{def: PHI}}.
    \end{itemize}
    Moreover, the constants appearing in $(1)$ and $(2)$ are quantitative in terms of those in $(2)$ and $(1)$, respectively.
\end{theorem}
Our proof of \Cref{thm: main thm} begins in \Cref{section: VD + PI implies PHI}, by presenting a proof of the parabolic Harnack inequality \eqref{eq: PHI} starting from volume doubling and a $(1,p)$-Poincar\'e inequality. We give a self-contained proof using De Giorgi--Di Benedetto-type iteration techniques and the expansion of positivity. The final argument for showing the parabolic Harnack inequality uses a variant of the technique used for the parabolic $p$-Laplace equation in \cite{DiBenedettoGianazzaVespri2008}. Another contribution is that we only assume the solution in \eqref{eq: PHI} to be nonnegative, whereas previous sources in the metric setting \cite{KinnunenKuusi2007, MarolaMasson} require that the solution be locally bounded away from zero. This is non-trivial since constants cannot generally be added to solutions of \eqref{eq: PDE} to produce a new solution.

The assumptions on reflexivity and the Rellich--Kondrachov compactness property in 
\Cref{thm: main thm} are only applied to prove the existence of a solution to  a Cauchy problem for \eqref{eq: PDE} in \Cref{section: existence}. These assumptions hold in geodesically complete Riemannian manifolds, and in the metric setting they are implied by doubling and Poincar\'e, see \cite{Cheeger,HajlaszKoskela}. Our approach to solutions of \eqref{eq: PDE} is variational, and the existence proof is based on the well-known minimizing movements scheme, adapted to our setting. We refer to \cite{BogeleinDuzaarMarcelliniScheven} and the references therein for existence results in the Euclidean setting, and \cite{surig2026} for existence in Riemannian manifolds. 

The proof that the parabolic Harnack inequalities implies doubling and Poincar\'e will proceed in several steps, which we outline below. In \Cref{section: PHI implies VD} we prove that the parabolic Harnack inequality implies volume doubling. Since the equation \eqref{eq: PDE} is nonlinear, solutions do not admit a representation formula in terms of the fundamental solution. Consequently, the classical arguments used in the linear theory are not available. 
We show that solutions to a Cauchy problem with initial data approximating the Dirac delta satisfy estimates analogous to those of the heat kernel, and that these estimates imply the volume doubling. The key steps are to replace representation formulas with a careful application of Caccioppoli inequalities for truncated shifts of solutions to \eqref{eq: PDE}, and iterating the parabolic Harnack inequality to transfer estimates in a small scale to larger scales.

In \Cref{section: PHI implies Poincare}, we show that the parabolic Harnack inequality implies a Poincar\'e inequality. First, we show a Poincar\'e-type inequality for functions with zero boundary values on a given ball in \Cref{prop: PHI implies local sobolev}. To obtain this inequality, we use the parabolic Harnack inequality to conclude quantitative decay of the $L^p$ norm of solutions to a Cauchy problem and apply the energy decreasing properties of \eqref{eq: PDE}. Then we present a technique of proving a general Poincar\'e inequality from an inequality for functions with zero boundary values, by applying the Dirichlet problem for the $p$-Laplace equation. The $p$-Laplace is the stationary version of \eqref{eq: PDE}, and hence the elliptic Harnack inequality is available for nonnegative solutions. The elliptic Harnack inequality is applied to localize functions to smaller balls, and to obtain H{\" o}lder continuity estimates. 

As a by-product of our proof, we obtain that the conjunction of volume doubling, a Poincar\'e-type inequality for zero boundary values and the elliptic Harnack inequality is equivalent with volume doubling and a Poincar\'e inequality, and in turn with the parabolic Harnack inequality. This equivalence was shown for $p=2$ in \cite{HebischSaloff-Coste2001}, but their proof proceeds via the established equivalence of Gaussian heat kernel estimates with volume doubling and a Poincar\'e inequality, and hence is not applicable in our setting.

More general forms of parabolic Harnack inequalities, involving a time lag of order $\Psi(r)$, for the heat equation have also been considered. Such parabolic Harnack inequalities appear in the analysis of fractals, and have been shown to characterize volume doubling, a $\Psi$-Poincar\'e inequality and a $\Psi$-cutoff Sobolev inequality, see \cite{BarlowBass2004, BarlowBassKumagai}. Related techniques have played an important role in the characterization and stability theory of elliptic Harnack inequalities for the Laplace equation \cite{BarlowMurugan, KajinoMurugan, Bass}. We hope that the methods developed in this work will be useful to study the corresponding nonlinear problems as well.

As an immediate consequence of \Cref{thm: main thm}, we obtain that the parabolic Harnack inequality \eqref{eq: PHI} is stable under perturbations of the underlying space which preserve volume doubling and Poincar\'e inequality.  \Cref{thm: main thm} also implies stability of parabolic Harnack inequalities with respect to the parameter $p$. By a result of Keith and Zhong \cite{KeithZhong}, a $(1,p)$-Poincar\'e inequality implies a $(1,q)$-Poincar\'e inequality for some $q<p$. Consequently, if the parabolic Harnack inequality \eqref{eq: PHI} holds for some $p$, it holds for some $q<p$.
We also note that a $(1,p)$-Poincar\'e inequality implies a $(1,q)$-Poincar\'e inequality for every $q>p$.
\section{Preliminaries}
\label{section: prelim}
\subsection{Geodesic spaces and doubling measures}
We consider a metric measure space $(X,d,\mu)$, where $X$ is a complete geodesic metric space with metric $d$, and $\mu$ is a Borel measure on $X$. A metric space $X$ is geodesic, if every pair of points in $X$ can be connected by a path with length equal to the distance between the points. We assume that 
the measure $\mu$ is nontrivial in the sense that $0<\mu(B(x_0,r))<\infty$
for every open ball $B(x_0,r)=\{x\in X:d(x,x_0)<r\}$, where $x_0\in X$ and $0<r<\infty$. The above assumptions will be in force throughout the rest of the article. Any additional assumptions  will be explicitly stated when they are imposed.

The measure $\mu$ is said to be doubling if there exists a constant $c_\mu\geq 1$ such that 
\begin{equation}
    \label{eq: doubling condition}
    \mu(B(x_0,2r))\leq c_\mu \mu(B(x_0,r))
\end{equation}
for every $x_0\in X$ and $r>0$.
If $\mu$ is doubling, there exist constants $Q>0$ and $c>0$, depending only on $c_\mu$, such that
\begin{equation}
    \label{eq: doubling implies radii control from below}
    \frac{\mu(B(x_0,r))}{\mu(B(x_0,R))}\geq c \left(\frac{r}{R}\right)^Q
\end{equation}
for every $x_0\in X$ and $0<r<R<\infty$. In fact, we may choose $c=c_\mu$ and $Q=\log_2c_\mu$.
The doubling condition also gives an upper bound for the ratio of the measure of balls.
If $X$ is connected and $\mu$ is doubling, then there exist constants $c>0$ and $\alpha>0$, depending only on $c_\mu$, such that
\begin{equation}
    \label{eq: doubling implies radii control from above}
    \frac{\mu(B(x_0,r))}{\mu(B(x_0,R))}\leq  c \left(\frac{r}{R}\right)^\alpha
\end{equation}
for every $x_0\in X$ and $0<r<R<\frac{1}{3}\diam X$. In particular, this holds in a geodesic space.
In a geodesic space, the doubling condition improves to the following annular decay property.
If $X$ is geodesic and $\mu$ is doubling, then there exist constants $c>0$ and $\beta>0$, depending only on $c_\mu$, such that 
\begin{equation}\label{eq: geodesic doubling becomes annular decay}
    \mu(B(x_0,r)\backslash B(x_0,(1-\delta)r))\leq c \delta^\beta \mu(B(x_0,r))
\end{equation}
for every $x_0\in X$, $r>0$ and $0<\delta<1$.

\subsection{Upper gradients and Sobolev spaces}
A nonnegative Borel function $g$ on $X$ is called an upper gradient
of an extended real valued function $u$ on $X$ if for all paths $\gamma$ joining points $x$ and $y$ in $\Omega$ we have
\begin{equation} \label{ug-cond}
|u(x)-u(y)|\le \int_\gamma g\,ds,
\end{equation}
whenever both $u(x)$ and $u(y)$ are finite, and $\int_\gamma g\,ds=\infty$ otherwise. 
We say that $g$ is a $p$-weak upper gradient of $u$
if the collection $\Gamma$ of rectifiable curves for which \eqref{ug-cond} fails has zero $p$-modulus, that is, there exists a nonnegative Borel function $f\in L^p(X;\mu)$
such that the integral $\int_\gamma f\, ds$ is infinite for every $\gamma\in \Gamma$. 

Let $\Omega$ be an open subset of $X$ and $1\leq p<\infty$. Let 
\begin{equation}
    \label{eq: newtonian space norm def}
    \Vert u\Vert_{N^{1,p}(\Omega)}=\Vert u\Vert_{L^{p}(\Omega)}+\inf\Vert g\Vert_{L^{p}(\Omega)},
\end{equation}
where the infimum is taken over all upper gradients $g$ of $u$ and 
define $\widetilde{N}^{1,p}(\Omega)$ as the space of functions $u$ for which $\Vert u\Vert_{N^{1,p}(\Omega)}<\infty$.
The Newton--Sobolev space is defined as the quotient space
\[
N^{1,p}(\Omega)=\lbrace u:\Vert u\Vert_{N^{1,p}(\Omega)}<\infty\rbrace/\sim,
\]
where $u\sim v$ if and only if $\Vert u-v\Vert_{N^{1,p}(\Omega)}=0$. We define $N^{1,q}_0(\Omega)$ as the space of functions $u\in N^{1,q}(X)$ that vanish $\mu$-almost everywhere in $X\setminus\Omega$. 

If $u$ has an upper gradient $g\in L^p(\Omega)$, there exists a unique minimal $p$-weak upper gradient $g_u\in L^p(\Omega)$ satisfying
$g_u\le g$ $\mu$-almost everywhere for all $p$-weak upper gradients $g\in L^p(\Omega)$ of $u$, see \cite[Theorem 6.3.20]{HeinonenKoskelaShanmugalingam}. Moreover, the minimal $p$-weak upper gradient is unique up to sets of measure zero.
For $u\in N^{1,p}(\Omega)$ we have
\[
\Vert u\Vert_{N^{1,p}(\Omega)}=\Vert u\Vert_{L^{p}(\Omega)}+\Vert g_u\Vert_{L^{p}(\Omega)},
\]
where $g_u$ is the minimal $p$-weak upper gradient of $u$.
A key advantage of $p$-weak upper gradients is that they behave better under $L^p$-convergence than upper gradients, see \cite[Proposition 2.2]{BjornBjorn2011}.
The distinction between upper gradients and $p$-weak upper gradients is relatively minor, since every $p$-weak upper gradient can be approximated by a sequence of upper gradients in $L^p$, see \cite[Lemma 6.2.2]{HeinonenKoskelaShanmugalingam}.
Consequently, the norm remains unchanged if the infimum in \eqref{eq: newtonian space norm def} is taken over $p$-weak upper gradients instead of upper gradients.

We collect some calculus rules for upper gradients on metric measure spaces. 
Let $u,v\in N^{1,p}(\Omega)$ and let $g_{u},\,g_{v}\in L^{p}(\Omega)$ be their minimal $p$-weak upper gradients. 
Then $g_{u}+g_{v}$ and $\vert u\vert g_{v}+\vert v\vert g_{u}$ are $p$-weak upper gradients for $u+v$ and $uv$, respectively, see \cite[Theorem 2.15]{BjornBjorn2011}.
Let  $\eta$ be Lipschitz function on $\Omega$ with $0\leq\eta\leq 1$ and consider $w=u+\eta(u-v)=(1-\eta)u+\eta v$. Then
$(1-\eta)g_{u}+\eta g_{v}+\vert v-u\vert g_{\eta}$
is a $p$-weak upper gradient of $w$, see \cite[Theorem 2.18]{BjornBjorn2011}.
Moreover, $g_{u}=g_{v}$, $\mu$-almost everywhere on the set $\lbrace x\in X: u(x)=v(x)\rbrace$.
In particular, if $c\in\mathbb{R}$ is a constant, then $g_{u}=0$ $\mu-$almost everywhere on the set $\lbrace x\in X: u(x)=c\rbrace$, see \cite[Corollary 2.21]{BjornBjorn2011}.

\subsection{Poincar\'e inequalities}
The integral average over a $\mu$-measurable subset $A$ of $X$, with $0<\mu(A)<\infty$, is denoted by
\[
u_{A}=\dashint_{A} u\dmu
=\frac{1}{\mu(A)}\int_{A}u\dmu.
\]

\begin{definition}\label{def: poincare}
Let $1\le p,q<\infty$. We say that a  $(q,p)$-Poincar\'e inequality holds in $(X,d,\mu)$ if there exist constants $c_P>0$ and $\tau\geq 1$ such that 
\begin{equation*}
    \biggl(\dashint_{B(x_0,r)}\lvert u-u_{B(x_0,r)}\rvert^q\dmu\biggr)^{\frac{1}{q}}\leq c_P r \biggl(\dashint_{B(x_0,\tau r)}  g_u^p\dmu\biggr)^\frac{1}{p}
\end{equation*}
for every $u\in N^{1,p}(B(x_0,\tau r))$, where $x_0\in X$ and $0<r<\infty$.
\end{definition}
We note that, in a geodesic space, a  $(q,p)$-Poincar\'e inequality with $\tau>1$ implies a  $(q,p)$-Poincar\'e inequality with $\tau=1$, see \cite[Theorem 4.34]{BjornBjorn2011}. When applying the Poincar\'e inequality in  \Cref{section: VD + PI implies PHI} we will apply it with $\tau=1$. However, when proving a $(1,p)$-Poincar\'e inequality in \Cref{section: PHI implies Poincare} we will prove it for a parameter $\tau>1$.

It is well known that the  Poincar\'e inequality is self-improving. The following lemma concerns the self-improvement of the left-hand side, see \cite[Theorem 9.1.15]{HeinonenKoskelaShanmugalingam} or \cite[Theorem 4.16]{BjornBjorn2011}.
\begin{lemma}
\label{lemma: self improvement of LHS in Poincaŕe}
Assume $\mu$ is doubling and that a $(1,p)$-Poincar\'e inequality holds in $(X,d,\mu)$ and let $Q$ be as in \eqref{eq: doubling implies radii control from below}. If $1\le p<Q$, then a  $(p^*,p)$-Poincar\'e inequality holds in $(X,d,\mu)$, where $p^*=\frac{pQ}{Q-p}$. If $p\geq Q$, then a  $(p^*,p)$-Poincar\'e inequality holds in $(X,d,\mu)$ for every $1\le p^*<\infty$.
\end{lemma}
Moreover, by a result of Keith and Zhong \cite{KeithZhong}, a Poincar\'e inequality also self-improves with respect to the right-hand side. 
\begin{lemma}
Assume that $\mu$ is doubling and let $1<p<\infty$. If a  $(1,p)$-Poincar\'e inequality holds in $(X,d,\mu)$, then there exists $q\ge1$, with $q<p$, depending on $c_\mu$, $c_P$ and $p$, such that a  $(1,q)$-Poincar\'e inequality holds in $(X,d,\mu)$.
\end{lemma}
The following lemma shows that a Poincar\'e inequality in balls with small radii compared to the diameter of the space implies the Poincar\'e inequality in all balls.
\begin{lemma}
\label{lemma: Poincare for smaller balls implies poincare}
Assume that $\mu$ is doubling and let $1\le p,q<\infty$. Assume that there exist constants $c$, $\tau'$ and $\delta>0$ such that 
\begin{equation}
    \label{eq: poincare for small balls}
    \biggl(\dashint_{B(x_0,r)}\lvert u-u_{B(x_0,r)}\rvert^q\dmu\biggr)^{\frac{1}{q}}\leq c r \biggl(\dashint_{B(x_0,\tau' r)}  g_u^p\dmu\biggr)^\frac{1}{p}
\end{equation}
for every $u\in N^{1,p}(B(x_0,\tau' r))$, where $x_0\in X$ and $0<r\leq\delta \diam(X)$. Then a  $(q,p)$-Poincar\'e inequality holds in $(X,d,\mu)$, with constants $c_P$ and $\tau$ depending only on $c$, $\tau'$, $\delta$ and $c_\mu$.
\end{lemma}
\begin{proof}
 We may assume that $0<r\leq \diam (X)$. Applying the standard covering theorem, see \cite[Lemma 1.7]{BjornBjorn2011}, and the doubling condition, there exists a finite collection of disjoint balls $B(x_i,\frac{\delta}{10}r)$, $i\in I=\{0,\dots, k\}$ with $B(x_0,r)\subset\cup_{i\in I} B(x_i,\frac{\delta}{2}r)$, $x_i\in B(x_0,r)$. Moreover, the number $k$ depends only on $\delta$ and $c_\mu$. Then 
\begin{equation*}
\begin{split}
    \biggl(\dashint_{B(x_0,r)} \lvert u - u_{B(x_0,r)}\rvert^q\dmu\biggr)^\frac{1}{q} & \leq 2 \biggl(\dashint_{B(x_0,r)} \lvert u - u_{B_0}\rvert^q\dmu\biggr)^\frac{1}{q} \\
    & \leq 2\sum_{i=0}^k \biggl(\dashint_{B(x_i,\frac{\delta}{2}r)} \lvert u-u_{B_0}\rvert^q \dmu\biggr)^\frac{1}{q}.
\end{split}
\end{equation*}
Let $i\in \{0,\dots,k\}$. Let $\gamma$ be a geodesic path connecting $x_i$ and $x_0$, and let $m=\lceil \frac{2}{\delta} \rceil$. Let $x_{i,j}$, $j\in \{1,\dots, m \}$, be such that $x_{i,1}=x_i$, $x_{i,m}=x_0$ and $d(x_{i,j},x_{i,j+1})\leq \frac{\delta}{2}r$. Then we have 
\begin{equation*}
\begin{split}
    \biggl(\dashint_{B(x_i,\frac{\delta}{2}r)} \lvert u-u_{B_0}\rvert^q \dmu\biggr)^\frac{1}{q} 
    & \leq \sum_{j=1}^{m-1} \lvert u_{B(x_{i,j},\frac{\delta}{2}r)}-u_{B(x_{i,j+1},\frac{\delta}{2}r)}\rvert\\
    &\qquad+ \biggl(\dashint_{B(x_i,\frac{\delta}{2}r)} \lvert u-u_{B(x_i,\frac{\delta}{2}r)}\rvert^q\dmu\biggr)^\frac{1}{q}
\end{split}
\end{equation*}
To estimate the first term on the right-hand side, we note that
\begin{equation*}
\begin{split}
    &\lvert u_{B(x_{i,j},\frac{\delta}{2}r)}-u_{B(x_{i,j+1},\frac{\delta}{2}r)}\rvert\\
    &\qquad\leq \lvert u_{B(x_{i,j},\frac{\delta}{2}r)}-u_{B(x_{i,j},\delta r)}\rvert +  \lvert u_{B(x_{i,j},\delta r)}-u_{B(x_{i,j+1},\frac{\delta}{2} r)}\rvert \\
    &\qquad\leq c \dashint_{B(x_{i,j},\delta r)} \lvert u-u_{B(x_{i,j},\delta r)}\rvert\dmu.
\end{split}
\end{equation*}
It follows that
\begin{equation*}
    \biggl(\dashint_{B(x_i,\frac{\delta}{2} r)} \lvert u-u_{B_0}\rvert^q \dmu\biggr)^\frac{1}{q} \leq c\sum_{j=1}^{m}\biggl(\dashint_{B(x_{i,j},\delta r)} \lvert u-u_{B(x_{i,j},\delta r)}\rvert^q\dmu\biggr)^\frac{1}{q}.
\end{equation*}
By \eqref{eq: poincare for small balls} we arrive at
\begin{equation*}
\begin{split}
    \biggl(\dashint_{B(x_0,r)} \lvert u - u_{B(x_0,r)}\rvert^q\dmu\biggr)^\frac{1}{q} & \leq c\sum_{i=0}^k\sum_{j=1}^{m}\biggl(\dashint_{B(x_{i,j},\delta r)} \lvert u-u_{B(x_{i,j},\delta r)}\rvert^q\dmu\biggr)^\frac{1}{q} \\
    & \leq c\sum_{i=0}^k\sum_{j=1}^{m} r\biggl(\dashint_{B(x_{i,j},\tau'\delta r)} g_u^p\dmu\biggr)^\frac{1}{p} \\
    & \leq c r\biggl(\dashint_{B(x_0,(1+\tau' \delta)  r)} g_u^p\dmu\biggr)^\frac{1}{p}.
\end{split}
\end{equation*}
In the final step we used the fact that the balls $B(x_{i,j}, \tau' \delta r)$ have bounded overlap, depending only on $c_\mu,\tau$ and $\delta$.
\end{proof}

\begin{lemma}
\label{lemma: charact of poinc}
Assume that $\mu$ is doubling. Let $1\le p<\infty$ and $1<q<\infty$. Assume that there exist strictly positive constants $c$, $\tau'$, $\epsilon$ and $\delta$, with $(1+\epsilon)(\frac{1}{2})^{1-\frac{1}{q}}<1$, such that
\begin{equation}
    \label{eq: charact of Poincare as reverse Holder}
    \biggl(\dashint_{B(x_0,r)}u^q \dmu\biggr)^\frac{1}{q} \leq c r \biggl(\dashint_{B(x_0,\tau' r)} g_u^p\dmu\biggr)^\frac{1}{p} + (1+\epsilon)\dashint_{B(x_0,r)}u\dmu
\end{equation}
for every nonnegative $u\in N^{1,p}(B(x_0,\tau' r))$, where $x_0\in X$ and $0<r\leq \delta \diam (X)$. Then a  $(q,p)$-Poincar\'e inequality holds in $(X,d,\mu)$, with constants $c_P$ and $\tau$, depending on $q$, $p$, $c$, $\tau'$, $\epsilon$ and $\delta$.
\end{lemma}
\begin{proof}
In order to be able to apply \Cref{lemma: Poincare for smaller balls implies poincare}, let $u\in N^{1,p}(B(x_0,\tau r))$ where $0<r\leq \delta \diam(X)$. Let $m$ be such that 
\[
\mu(\{ u\geq m\}\cap B(x_0,r)) \geq \frac{1}{2}\mu(B(x_0,r))
\]
and  
\[
\mu(\{ u \leq  m\}\cap B(x_0,r)) \geq \frac{1}{2}\mu(B(x_0,r)).
\]
Applying \eqref{eq: charact of Poincare as reverse Holder} to $(u-m)_+$ gives
\begin{equation}
\label{eq: poincare for u-m pos}
\begin{split}
    \biggl(\dashint_{B(x_0,r)}(u-m)_+^q \dmu\biggr)^\frac{1}{q} 
    &\leq c_P r \biggl(\dashint_{B(x_0,\tau r)} g_u^p\dmu\biggr)^\frac{1}{p}\\
    &\qquad+ (1+\epsilon)\dashint_{B(x_0,r)}(u-m)_+\dmu.
\end{split}
\end{equation}
Here we used the fact that $g_{(u-k)_+}\leq g_u$. By H\"older's inequality we have 
\begin{equation*}
\begin{split}
    &(1+\epsilon)\dashint_{B(x_0,r)}(u-m)_+\dmu\\
    &\qquad \leq (1+\epsilon) \biggl( \frac{\mu(\{u>m\}\cap B(x_0,r))}{\mu(B(x_0,r))}\biggr)^{1-\frac{1}{q}}\biggl(\dashint_{B(x_0,r)}(u-m)_+^q\dmu\biggr)^\frac{1}{q} \\
    &\qquad \leq (1+\epsilon) \Big( \frac{1}{2}\Big)^{1-\frac{1}{q}}\biggl(\dashint_{B(x_0,r)}(u-m)_+^q\dmu\biggr)^\frac{1}{q}.
\end{split}
\end{equation*}
Let $\gamma=1-(1+\epsilon)(\frac{1}{2})^{1-\frac{1}{q}}$. By assumption $\gamma>0$. Applying the previous estimate in \eqref{eq: poincare for u-m pos} and absorbing terms we find that 
\begin{equation*}
    \biggl(\dashint_{B(x_0,r)}(u-m)_+^q \dmu\biggr)^\frac{1}{q} \leq \frac{c_P}{\gamma} r \biggl(\dashint_{B(x_0,\tau r)} g_u^p\dmu\biggr)^\frac{1}{p}.
\end{equation*}
A similar argument shows that 
\begin{equation*}
    \biggl(\dashint_{B(x_0,r)}(m-u)_+^q \dmu\biggr)^\frac{1}{q} \leq \frac{c_P}{\gamma} r \biggl(\dashint_{B(x_0,\tau r)} g_u^p\dmu\biggr)^\frac{1}{p}.
\end{equation*}
Combining the estimates above gives
\begin{equation*}
    \biggl(\dashint_{B(x_0,r)}\lvert u-m\rvert^q \dmu\biggr)^\frac{1}{q} \leq \frac{2 c_P}{\gamma} r \biggl(\dashint_{B(x_0,\tau r)} g_u^p\dmu\biggr)^\frac{1}{p}.
\end{equation*}
Finally, we conclude that
\begin{equation*}
\begin{split}
    \biggl(\dashint_{B(x_0,r)}\lvert u-u_{B(x_0,r)}\rvert^q\dmu\biggr)^{\frac{1}{q}} & \leq 2 \biggl(\dashint_{B(x_0,r)}\lvert u-m\rvert^q\dmu\biggr)^{\frac{1}{q}} \\
    & \leq \frac{4 c_P}{\gamma} r \biggl(\dashint_{B(x_0,\tau r)} g_u^p\dmu\biggr)^\frac{1}{p}.
\end{split}
\end{equation*}
The claim follows from \Cref{lemma: Poincare for smaller balls implies poincare}.
\end{proof}
The previous lemma gives a characterization of the Poincar\'e inequality, in the sense that if a  $(q,p)$-Poincar\'e holds in $(X,d,\mu)$, then
\begin{equation*}
\begin{split}
    \biggl(\dashint_{B(x_0,r)} u^q \dmu  \biggr)^\frac{1}{q} & \leq \biggl(\dashint_{B(x_0,r)}\lvert u-u_{B(x_0,r)}\rvert^q\dmu\biggr)^{\frac{1}{q}} + \int_{B(x_0,r)} u \dmu \\
    & \leq c_P r \biggl(\dashint_{B(x_0,\tau r)} g_u^p\dmu\biggr)^\frac{1}{p} + \dashint_{B(x_0,r)}u\dmu,
\end{split}
\end{equation*}
so that \eqref{eq: charact of Poincare as reverse Holder} holds with $\epsilon =0$. This observation implies the following result for functions that vanish on a large portion of a set, see \cite[Lemma 2.1]{KinnunenShanmugalingam}.
\begin{lemma}
\label{eq: Poincare for function zero on large set}
Assume that $\mu$ is doubling. Let $1\le p<\infty$ and $1<q<\infty$. Assume that a  $(q,p)$-Poincar\'e inequality holds in $(X,d,\mu)$. Let $u\in N^{1,p}(X)$ and assume that there exists $0<\gamma<1$ such that $\mu(\{ \lvert u\rvert >0\}\cap B(x_0,r)) \leq \gamma \mu(B(x_0,r))$, where $x_0\in X$ and $0<r<\infty$. Then there exists a constant $c$, depending on $c_\mu$, $c_P$, $p$, $q$ and $\gamma$, such that
\begin{equation*}
    \biggl( \dashint_{B(x_0,r)} \lvert u\rvert^{q} \dmu \biggr)
^\frac{1}{q}\leq c r \biggl( \dashint_{B(x_0, \tau r)} g_u^p\dmu \biggr)^\frac{1}{p}.
\end{equation*}
\end{lemma}
The previous lemma implies the following Sobolev inequality for functions with zero boundary values, see \cite{KinnunenShanmugalingam}.
\begin{lemma}
\label{lemma: sobolev from poincare}
Assume that $\mu$ is doubling. Let $1\le p<\infty$ and $1<q<\infty$.
Assume that a $(q,p)$-Poincar\'e inequality holds in $(X,d,\mu)$. Then there exists a constant $c$, depending on $c_\mu$, $c_P$, $p$ and $q$, such that
\begin{equation*}
    \biggl(\dashint_{B(x_0,r)} \lvert u\rvert^q \dmu\biggr)^\frac{1}{q} \leq cr \biggl(\dashint_{B(x_0,r)} g_u^p \dmu\biggr)^\frac{1}{p}
\end{equation*}
for every $u\in N^{1,p}_0(B(x_0,r))$, where $x_0\in X$ and $0<r\leq\frac13\diam(X)$.
\end{lemma}

\subsection{Parabolic Sobolev spaces}
We briefly discuss time dependent Sobolev spaces that are relevant to the doubly nonlinear equation \eqref{eq: PDE}.
\begin{definition}\label{def.pnewtonian}
Let $\Omega$ be an open subset of $X$, $0<T<\infty$ and $1\le p<\infty$. 
The parabolic Newton--Sobolev space $L^{p}(0, T;N^{1,p}(\Omega))$ consists of  strongly measurable functions 
$u:(0,T)\to N^{1,p}(\Omega)$ with the norm
\begin{equation*}
\Vert u\Vert_{L^{p}(0, T;N^{1,p}(\Omega))}
=\biggl(\int_{0}^{T}\Vert u(t)\Vert_{N^{1,p}(\Omega)}^{p}\dt\biggr)^{\frac1p}<\infty.
\end{equation*}
The integration over time is taken with respect to the one-dimensional Lebesgue measure. It is also possible to consider parabolic Newton--Sobolev spaces over other time intervals, for example, $L^{p}(t_1,t_2;N^{1,p}(\Omega))$, $-\infty<t_1<t_2<\infty$, or $L^{p}(-\infty,\infty;N^{1,p}(\Omega))$.
\end{definition}

The product measure in the space $X\times (0,T)$, $0<T<\infty$, is denoted by $\nu$.
If $u\in L^{p}(0, T;N^{1,p}(\Omega))$, it is immediate that $u\in L^{p}(0, T;L^{p}(\Omega))$. The space-time cylinder is denoted by $\Omega_{T}=\Omega\times(0, T)$. For $u\in L^{p}(0, T;L^{p}(\Omega))$, there exists a $\nu$-measurable representative $u:\Omega_T\to[-\infty,\infty]$ such that $u(t)=u(\cdot,t)$ for almost every $0<t<T$ and
\[
\int_{\Omega_T}|u|^p\dnu
=\int_0^T\int_\Omega|u(x,t)|^p\dmu\dt
=\int_0^T\Vert u(t)\Vert_{L^p(\Omega)}^p\dt.
\]
If $u\in L^{p}(0, T;N^{1,p}(\Omega))$, we may conclude that the map $t\mapsto g_{u(t)}$ belongs to $L^{p}(0, T;L^{p}(\Omega))$. Moreover, we find a $\nu$-measurable representative
$g_u:\Omega_T\to[-\infty,\infty]$ such that $g_u(t)=g_u(\cdot,t)$ for almost every $0<t<T$ and 
\[
\int_{\Omega_T}g^p\dnu
=\int_0^T\int_\Omega g_u(x,t)^p\dmu\dt
=\int_0^T\Vert g_u(t)\Vert_{L^p(\Omega)}^p\dt.
\]
We shall implicitly identify $u$ and $g_u$ with their $\nu$-measurable representatives.

The parabolic Newton--Sobolev space with zero lateral boundary values is defined by saying that $u\in L^p(0,T; N^{1,p}_0(\Omega))$ if $u\in L^{p}(0, T;N^{1,p}(\Omega))$ and $u(t)\in N^{1,p}_0(\Omega)$ for almost every $0<t<T$. The following is a parabolic Sobolev embedding for functions $u\in L^p(0,T; N^{1,p}_0(\Omega))$.
\begin{lemma}
\label{lemma: par sob}
Assume that $\mu$ is doubling and that a $(1,p)$-Poincar\'e inequality holds in $(X,d,\mu)$. Then there exists $\kappa >1$, depending only on $p$ and $c_\mu$, and a constant $c>0$, depending on $c_\mu$, $c_P$ and $p$, such that 
\begin{equation*}
    \int_{0}^{T}\dashint_{B(x_0,r)} \lvert u\rvert^{p\kappa}\dmu\dt \leq c r^p \int_{0}^{T}\dashint_{B(x_0,r)} g_u^p\dmu\dt \esssup_{0<t<T}\biggl( \dashint_{B(x_0,r)} u^p\dmu \biggr)^{\kappa-1}
\end{equation*}
for every $u\in L^p(0,T;N^{1,p}_0(B(x_0,r)))\cap L^\infty(0,T;L^p(B(x_0,r)))$,
where $x_0\in X$ and $0<r\leq\frac13\diam(X)$.
\end{lemma}
\begin{proof}
Let $\kappa=2-\frac{p}{p^*}$, where $p^*$ is as in \Cref{lemma: self improvement of LHS in Poincaŕe} and let $0<t<T$. H\"older's inequality and \Cref{lemma: sobolev from poincare} implies
\begin{equation*}
\begin{split}
    \dashint_{B(x_0,r)} \lvert u\rvert^{p\kappa}\dmu & =  \dashint_{B(x_0,r)} \lvert u\rvert^p \lvert u\rvert^{p(\kappa-1)}\dmu \\
    & \leq \biggl( \dashint_{B(x_0,r)} \lvert u\rvert^{p^*} \dmu\biggr)^\frac{p}{p^*} \biggl(\dashint_{B(x_0,r)} \lvert u\rvert^{p}\dmu\biggr)^{\kappa-1} \\
    & \leq cr^p \dashint_{B(x_0,r)} g_u^p \dmu \biggl(\dashint_{B(x_0,r)} \lvert u\rvert^{p}\dmu\biggr)^{\kappa-1}.
\end{split}
\end{equation*}
The claim follows by integrating the above estimate over $0<t<T$.
\end{proof}

\section{Doubly nonlinear equation}

In this section we study variational solutions to a doubly nonlinear equation and derive related energy estimates. We emphasize that neither the doubling condition nor a Poincar\'e inequality is required in this section. We use the convention that the power of a signed number $b$ is given by
\begin{equation}\label{eq: signconv}
    b^\alpha = \begin{cases}
            \lvert b\rvert^{\alpha-1}b ,&\quad b \neq 0,\\
            0, & \quad b=0. 
        \end{cases}
\end{equation}

We begin with the definition of a variational solution of the doubly nonlinear equation \eqref{eq: PDE}. At this stage, we fix the parameter $1<p<\infty$ for the rest of the article.
\begin{definition}\label{def.soldne}
Let $\Omega$ be an open subset of $X$, $0<T<\infty$ and $1<p<\infty$.
A function $u\in L^p(0,T;N^{1,p}(\Omega))$ is a  solution of \eqref{eq: PDE}
in $\Omega_T=\Omega\times(0,T)$ if 
\begin{equation}
\label{eq: var ineq}
    p\int_{\Omega_T}u^{p-1}\partial_t{\varphi}\dnu  + \int_{\Omega_T} g_u^p\dnu \leq \int_{\Omega_T} g_{u+\varphi}^p\dnu,
\end{equation}
for every $\varphi\in L^p(0,T;N^{1,p}(\Omega))$ with $\partial_t \varphi\in L^p(\Omega_T)$ and  $\supp \varphi\subset\Omega_T$. If \eqref{eq: var ineq} holds for all such $\varphi$ with $\varphi\geq 0$ we say that $u$ is a  supersolution. 
If \eqref{eq: var ineq} holds for all such $\varphi$ with $\varphi\leq 0$ we say that $u$ is a subsolution.
\end{definition}

The following direct consequence of the definition of a solution will be useful.
\begin{lemma}
\label{lemma: var sol implies weak sol with inequality}
Assume that $u$ is a solution of \eqref{eq: PDE} in $\Omega_T$. Then
\begin{equation*}
    \int_{\Omega_T} u^{p-1}\partial_t\varphi \dnu \leq \int_{\Omega_T} g_u^{p-1}g_\varphi\dnu 
\end{equation*}
for every $\varphi\in L^p(0,T;N^{1,p}(\Omega))$ with $\partial_t \varphi\in L^p(\Omega_T)$ and  $\supp \varphi\subset\Omega_T$.
\end{lemma}
\begin{proof}
Let $h>0$ and test \eqref{eq: var ineq} with $h\varphi$ to get
\begin{equation*}
    p\int_{\Omega_T} u^{p-1} \partial_t{\varphi}\dnu  \leq \frac{1}{h}\int_{\Omega_T}(g_{u+h\varphi}^p-g_u^p)\dnu
\end{equation*}
The mean value theorem implies that
\begin{equation*}
    g_{u+h\varphi}^p-g_u^p\leq (g_u +h g_{\varphi})^p-g_u^p \leq p (g_u + hg_\varphi)^{p-1}hg_\varphi,
\end{equation*}
and hence 
\begin{equation*}
    p\int_{\Omega_T}u^{p-1}\partial_t{\varphi}\dnu  \leq p\int_{\Omega_T} (g_u + hg_\varphi)^{p-1}g_\varphi\dnu.
\end{equation*}
The claim follows by letting $h\to 0$.
\end{proof}

\subsection{Time mollification}
In general, the time derivative of a weak subsolution or supersolution does not
belong to $L^p(\Omega_T)$ and does not even exist a priori. To be able to derive a suitable energy estimate,
we use the following mollification in time.
Let $u\in L^p(0,T;N^{1,p}(\Omega))$. For $\delta>0$, we define 
\begin{equation*}
    u_{\delta}(t) =\frac{1}{\delta}\int_0^te^{\frac{s-t}{\delta}}u(x,s)\ds
\end{equation*}
For the following lemma we refer to \cite[Lemma B.1]{BogeleinDuzaarMarcellini}.
\begin{lemma}
Let $u\in L^p(0,T;N^{1,p}(\Omega))$.
Then $u_{\delta}\in L^p(0,T;N^{1,p}(\Omega))$ and 
\begin{equation*}
    \|u_\delta\|_{L^p(0,T;N^{1,p}(\Omega))} \leq \|u\|_{L^p(0,T;N^{1,p}(\Omega))}
\end{equation*}
for every $\delta>0$.
Moreover, we have $\partial_t u_\delta\in L^p(0,T;N^{1,p}(\Omega))$ and 
\begin{equation}
\label{eq: time der of mollification}
    \partial_t u_\delta = \frac{1}{\delta}(u-u_\delta)
\end{equation}
for every $\delta>0$.
\end{lemma}
We also have the convergence in norm and convergence of mixed terms, as in the following lemma.
\begin{lemma}
\label{lemma: convergence of time mollification}
Let $u\in L^p(0,T;N^{1,p}(\Omega))$, $1\le p<\infty$. Then 
\begin{equation*}
    \lim_{\delta\to 0} \| u_\delta -u\|_{L^p(0,T;N^{1,p}(\Omega))}=0,
\end{equation*}
and
\begin{equation*}
    \lim_{\delta \to 0} \frac{1}{\delta}\int_0^T \int_0^{T-t}\| u(t)-u(t+s)\|_{N^{1,p}(\Omega)}^p e^{-\frac{s}{\delta}}\ds\dt=0.
\end{equation*}
\end{lemma}
\begin{proof}
We only prove the second claim, since the proof of the first claim is similar. Let $\epsilon>0$. Since $C([0,T];N^{1,p}(\Omega))$ is dense in $L^p(0,T;N^{1,p}(\Omega))$, we may take a function $v\in C([0,T];N^{1,p}(\Omega))$ with $\|u-v\|_{L^p(0,T;N^{1,p}(\Omega))}<\epsilon$. Then 
\begin{equation*}
\begin{split}
    &\frac{1}{\delta}\int_0^T \int_0^{T-t}\| u(t)-u(t+s)\|_{N^{1,p}(\Omega)}^p e^{-\frac{s}{\delta}}\ds\dt \\
    &\qquad\leq \frac{c}{\delta}\int_0^T \int_0^{T-t} \bigl( \| u(t)-v(t)\|_{N^{1,p}(\Omega)}^p+\|v(t)-v(t+s)\|_{N^{1,p}(\Omega)}^p\\ 
    &\qquad\qquad\qquad\qquad\qquad+\|v(t+s)-u(t+s)\|_{N^{1,p}(\Omega)}^p \bigr)e^{-\frac{s}{\delta}}\ds\dt \\
    &\qquad \leq c \|u-v\|_{L^p(0,T;N^{1,p}(\Omega))}^p + \frac{c}{\delta}\int_0^T \int_0^{T-t} \|v(t)-v(t+s)\|_{N^{1,p}(\Omega)}^p e^{-\frac{s}{\delta}}\ds\dt.
\end{split}
\end{equation*}
Let $M=\max_{t\in [0,T]} \| v(t)\|_{N^{1,p}(\Omega)}$. Since $v\in C([0,T];N^{1,p}(\Omega))$ there exists $\eta>0$ such that  $\| v(t_1) -v(t_2)\|_{N^{1,p}(\Omega)}\leq \epsilon$ whenever $0<t_1,t_2<T $ with $\lvert t_1-t_2\rvert<\eta$. To estimate the second term on the right-hand side of the estimate above, we split the inner integral into parts as 
\begin{equation*}
\begin{split}
    &\frac{c}{\delta}\int_0^T \int_0^{T-t} \|v(t)-v(t+s)\|_{N^{1,p}(\Omega)}^p e^{-\frac{s}{\delta}}\ds\dt\\
    & = \frac{c}{\delta}\int_0^T \int_0^{\eta} \|v(t)-v(t+s)\|_{N^{1,p}(\Omega)}^p e^{-\frac{s}{\delta}}\ds\dt \\
    &\qquad\qquad+ \frac{c}{\delta}\int_0^T \int_\eta^{T-t} \|v(t)-v(t+s)\|_{N^{1,p}(\Omega)}^p e^{-\frac{s}{\delta}}\ds\dt \\
    &\qquad \leq \frac{c \epsilon ^p}{\delta}\int_0^T\int_0^\eta e^{-\frac{s}{\delta}}\ds\dt + cM \int_0^T\int_\eta^{T-t}e^{-\frac{s}{\delta}} 
     \leq cT\epsilon^p + cMT e^{-\frac{\eta}{\delta}}.
\end{split}
\end{equation*}
It follows that
\begin{equation*}
    \frac{1}{\delta}\int_0^T \int_0^{T-t}\| u(t)-u(t+s)\|_{N^{1,p}(\Omega)}^p e^{-\frac{s}{\delta}}\ds\dt \leq c\epsilon^p + ce^{-\frac{\eta}{\delta}}.
\end{equation*}
The claim follws by letting first $\delta\to 0$ and then $\epsilon \to 0$.
\end{proof}
Though we defined $u_{\delta}$ only for $u\in L^p(0,T;N^{1,p}(\Omega))$, the definition makes sense more generally for $u\in L^p(0,T;L^{p}(\Omega))$. Convergence of $u^\delta$ in the norm of $L^p(0,T;N^{1,p}(\Omega))$ follows as in the previous lemma.

\subsection{Energy estimate}
Let
\begin{equation*}
    G_{\pm}(w,k)=\pm (p-1) \int_{k}^{w}\lvert s\rvert^{p-2}(s-k)_{\pm}\ds.
\end{equation*}
A change of variables shows that 
\begin{equation}
    \label{eq: primitive for energy est after change of var}
    G_\pm(w,k)=\int_{k^{p-1}}^{u^{p-1}}(s^\frac{1}{p-1}-k)_\pm\ds.
\end{equation}
The following lemma is from \cite[Lemma 2.2]{BogeleinDuzaarLiao}.
\begin{lemma}
There exists a constant $c$, depending on $p$, such that
\begin{equation*}
    \frac{1}{c}(\lvert w\rvert+\lvert k\rvert)^{p-2}(w-k)_\pm^2 \leq G_{\pm}(w,k)\leq c (\lvert w\rvert+\lvert k\rvert)^{p-2}(w-k)_\pm^2,
\end{equation*}
for every $w,k\in \mathbb R$.
\end{lemma}

The following energy estimate will be useful in the proof of the main result.

\begin{lemma}
\label{lemma: caccioppoli with cutoff}
Let $u$ be a subsolution of \eqref{eq: PDE} in $\Omega_T$. Then there exists a constant $c$, depending on $p$, such that, for any $k\in\mathbb R$ and almost every $0<t_0<T$, we have 
\begin{equation*}
\begin{split}
    & \esssup_{t_0<t<T}p\int_{\Omega} G_+(u(t),k)\varphi^p\dmu + \int_{\Omega_{T}}g_{(u-k)_+}^p\varphi^p\dnu \\
    &\qquad \leq c\int_{\Omega_{T}}\bigl((u-k)_+^p g_{\varphi}^p+G_+(u,k)(\partial_t\varphi^p)_+\bigr)\dnu + p\int_{\Omega} G_+ (u(t_0),k)\varphi^p\dmu 
\end{split}
\end{equation*}
for every nonnegative Lipschitz continuous function $\varphi$ with $\supp \varphi \subset \Omega_T$.
\end{lemma}
\begin{proof}
We may assume that $0\leq \varphi\leq 1$. Since the support of $\varphi$ is compact, after a change of variables, we may assume that $0$ is a Lebesgue point of $u$ in the sense that 
\begin{equation*}
    \lim_{h\to 0}\frac{1}{h}\int_{0}^{h} \int_{\Omega} \lvert u - u(0) \rvert^p\dmu\dt=0.
\end{equation*}
Let $\delta>0$. The mollified form of \eqref{eq: var ineq} reads
\begin{equation}
\label{eq: mollified var ineq}
\begin{split}
    -p\int_{\Omega_T} \partial_t(u^{p-1})_\delta \eta\dnu  + &\int_{\Omega_T} (g_u^p)_\delta \dnu
    \leq \frac{1}{\delta}\int_0^T\int_0^{T-t}\int_{\Omega} g_{u(t)+\eta(t+s)}^p\dmu\ds\dt\\
    &+\int_{\Omega}u(0)^{p-1}\frac{1}{\delta}\int_0^T e^{-\frac{s}{\delta}}\eta(x,s)\ds\dmu   ,
\end{split}
\end{equation}
for every $\eta\in L^p(0,T;N_0^{1,p}(\Omega))$. Let $0<\tau_1<\tau_2<T$ and $\epsilon>0$. Consider a cutoff function
\begin{equation*}
    \xi(t)
        =\begin{cases}
            0 ,&\quad 0<t\leq \tau_1-\epsilon\\
            \frac{1}{\epsilon}(t-\tau_1+\epsilon), & \quad \tau_1-\epsilon <t\leq \tau_1, \\
            1,&\quad \tau_1<t<\tau_2,\\
            \frac{1}{\epsilon}(\tau_2+\epsilon-t), & \quad \tau_2<t\leq \tau_2+\epsilon,\\
            0,&\quad \tau_2 +\epsilon <t<T .
        \end{cases}
\end{equation*}
We test \eqref{eq: mollified var ineq} with $\eta=-\xi\varphi^p  (u-k)_+$. For the term containing the time derivative, we have
\begin{equation*}
\begin{split}
    -p\int_{\Omega_T} \partial_t(u^{p-1})_\delta \eta\dnu & = p\int_{\Omega_T} \partial_t (u^{p-1})_\delta \xi\varphi^p(u-k)_+\dnu \\
    & = p\int_{\Omega_T} \partial_t(u ^{p-1})_\delta \xi\varphi^p( (u^{p-1})_\delta^{\frac{1}{p-1}}-k)_+\dnu \\
    & \qquad + p\int_{\Omega_T}\bigl(\partial_t(u^{p-1})_\delta \xi\varphi^p( (u-k)_+ -(u^{p-1})_\delta^{\frac{1}{p-1}}-k)_+\bigr)\dnu.
\end{split}
\end{equation*}
Using \eqref{eq: time der of mollification} and the fact that $\lambda\mapsto ( \lambda^\frac{1}{p-1}-k)_+$ is a monotonically increasing function, we obtain
\begin{equation*}
\begin{split}
    &p\int_{\Omega_T} \partial_t(u^{p-1})_\delta \xi\varphi^p( (u-k)_+ -(u^{p-1})_\delta^{\frac{1}{p-1}}-k)_+\dnu \\
    &\qquad= p\int_{\Omega_T}\Bigl( \frac{1}{\delta}(u^{p-1}-(u^{p-1})_\delta) \xi\varphi^p( (u-k)_+ -(u^{p-1})_\delta^{\frac{1}{p-1}}-k)_+\Bigr)\dnu
     \geq 0.
\end{split}
\end{equation*}
Thus we have
\begin{equation*}
\begin{split}
    -p\int_{\Omega_T} \partial_t(u^{p-1})_\delta \eta\dnu 
    &\geq p\int_{\Omega_T} \partial_t(u^{p-1})_\delta \xi\varphi^p( (u^{p-1})_\delta^{\frac{1}{p-1}}-k)_+\dnu \\
     & = - p\int_{\Omega_T} \partial_t(\varphi^p \xi) G_+((u^{p-1})_\delta^\frac{1}{p-1},k)\dnu.
\end{split}
\end{equation*}
In the second line we used \eqref{eq: primitive for energy est after change of var} to conclude that 
\begin{equation*}
    \partial_t G_+((u^{p-1})_\delta^\frac{1}{p-1},k)= \partial_t (u^{p-1})_\delta ((u^{p-1})_\delta^\frac{1}{p-1}-k)_+.
\end{equation*}
Using \eqref{eq: primitive for energy est after change of var} and the convergence $(u^{p-1})_\delta\to u^{p-1}$ in $L^{p'}$ as $\delta\to 0$, we also conclude that 
\begin{equation*}
    \lim_{\delta\to 0}\int_{\Omega_T} \partial_t(\varphi^p \xi) G_+((u^{p-1})_\delta^\frac{1}{p-1},k)\dnu = \int_{\Omega_T} \partial_t(\varphi^p \xi) G_+(u,k)\dnu.
\end{equation*}
Applying \Cref{lemma: convergence of time mollification} and the fact that $\eta$ has a compact support in $(0,T)$, we have
\begin{equation*}
\begin{split}
    &\lim_{\delta\to 0} \frac{1}{\delta}\int_0^T\int_0^{T-t}\int_{\Omega} g_{u(t)-\eta(t+s)}^p\dmu\ds\dt+\int_{\Omega} u(0)^{p-1}\frac{1}{\delta}\int_0^T e^{-\frac{s}{\delta}}\eta(x,s)\ds\dmu \\
    &\qquad= \int_{\Omega_T} g_{u-\eta}^p\dnu.
\end{split}
\end{equation*}
By letting $\delta\to 0$ in \eqref{eq: mollified var ineq}, we obtain
\begin{equation*}
\begin{split}
    & -p\int_{\Omega_T} \partial_t(\varphi^p \xi)G_+(u,k)\dnu + \int_{\Omega_T} g_{u}^p\dnu  \leq \int_{\Omega_T} g_{u-\eta}^p\dnu.
\end{split}
\end{equation*}
The product rule implies that
\begin{equation*}
\begin{split}
     &-p\int_{\Omega_T} \partial_t(\varphi^p \xi)G_+(u,k)\dnu + \int_{\Omega_T} g_{(u-k)_+}^p\dnu \leq  \int_{\Omega_T} g_{(1-\xi\varphi^p)(u-k)_+}^p\dnu \\
    &\qquad \leq \int_{\Omega_T}\bigl((1-\xi\varphi^p)g_{(u-k)_+}+\xi p\varphi^{p-1}g_\varphi (u-k)_+ \bigr)^p\dnu.
\end{split}
\end{equation*}
By the convexity of $t\mapsto t^p$, we have
\begin{equation*}
\begin{split}
    &\bigl((1-\xi\varphi^p)g_{(u-k)_+}+\xi p\varphi^{p-1}g_\varphi (u-k)_+ \bigr)^p\\ 
    &\qquad = \bigl((1-\xi\varphi^p)g_{(u-k)_+}+\xi\varphi^{p}\frac{p}{\varphi}g_\varphi (u-k)_+ \bigr)^p \\
    &\qquad \leq (1-\xi \varphi^p)g_{(u-k)_+}^p+\xi p^p g_\varphi^p(u-k)_+,
\end{split}
\end{equation*}
from which it follows that
\begin{equation*}
\begin{split}
    -p\int_{\Omega_T}&  (\partial_t \varphi^p \xi + \partial_t \xi \varphi^p)G_+(u,k)\dnu + \int_{\Omega_T} g_{(u-k)_+}^p\dnu \\
    &\leq \int_{\Omega_T} (1-\xi\varphi^p)g_{(u-k)_+}^p\dnu + p^p \int_{\Omega_T} \xi g_{\varphi}^p (u-k)_+^p\dnu.
\end{split}
\end{equation*}
Absorbing the first term on the right-hand side, we obtain
\begin{equation*}
\label{eq: energy estimate before time der calc}
\begin{split}
    &-p\int_{\Omega_T}  (\partial_t \varphi^p \xi + \partial_t \xi \varphi^p)G_+(u,k)\dnu + \int_{\Omega_T} g_{(u-k)_+}^p\xi\varphi^p\dnu\\
    &\qquad\leq p^p \int_{\Omega_T} g_{\varphi}^p (u-k)_+^p\dnu.
    \end{split}
\end{equation*}
By letting $\epsilon \to 0$, we conclude that
\begin{equation}
\label{eq: energy estimate before choice of times}
\begin{split}
    & p\int_{\Omega} G_+(u(\tau_2),k)\dmu + \int_{\Omega_{\tau_2,\tau_2}} g_{(u-k)_+}^p\varphi^p\dnu 
    \leq p^p \int_{\Omega_T} g_{\varphi}^p (u-k)_+^p\dnu\\
    &\qquad+ p\int_{\Omega_T} g(u,k)\partial_t(\varphi^p)\dnu+ p\int_{\Omega} G_+(u(\tau_1),k)\varphi^p\dmu,
\end{split}
\end{equation}
for almost every $0<\tau_1,\tau_2<T$. Letting $\tau_1\to 0$ and $\tau_2\to T$ in \eqref{eq: energy estimate before choice of times}, we obtain
\begin{equation}
\label{eq: energy estimate without esssup}
    \int_{\Omega_{T}} g_{(u-k)_+}^p\varphi^p\dnu\leq p^p \int_{\Omega_T} g_{\varphi}^p (u-k)_+^p\dnu + p\int_{\Omega_T} g(u,k)(\partial_t\varphi^p)_+\dnu.
\end{equation}
Discarding the second term on the left-hand side of \eqref{eq: energy estimate before choice of times} and letting $\tau_1=t_0$, $t_0<\tau_2<T$, we conclude that
\begin{equation}
\label{eq: esssup term for energy est}
\begin{split}
    & \esssup_{t_0<t<T} p\int_{\Omega} G_+(u(t),k)\dmu 
    \leq p^p \int_{\Omega_T} g_{\varphi}^p (u-k)_+^p\dnu\\
    &\qquad+ p\int_{\Omega_T} g(u,k)\partial_t(\varphi^p)\dnu+ p\int_{\Omega} G_+(u(t_0),k)\varphi^p\dmu.
\end{split}
\end{equation}
By adding up \eqref{eq: energy estimate without esssup} and \eqref{eq: esssup term for energy est}, we arrive at 
\begin{equation*}
\begin{split}
    & \esssup_{t_0<t<T}p\int_{\Omega} G_+(u(t),k)\varphi^p\dmu + \int_{\Omega_{T}}g_{(u-k)_+}^p\varphi^p\dnu \\
    & \qquad\leq c\int_{\Omega_{T}}\bigl((u-k)_+^p g_{\varphi}^p+G_+(u,k)(\partial_t\varphi^p)_+\bigr)\dnu + p\int_{\Omega} G_+ (u(t_0),k)\varphi^p\dmu.
\end{split}
\end{equation*}
This completes the proof.
\end{proof}

\section{Doubling and Poincar\'e imply parabolic Harnack}

In this section we prove that if $\mu$ is doubling and a $(1,p)$-Poincar\'e inequality holds in $(X,d,\mu)$, then nonnegative solutions of \eqref{eq: PDE} satisfy a scale and location invariant parabolic Harnack inequality.
For $r>0$, $x_0\in X$, $t_0\in \R$ and $\delta>0$, we denote 
\begin{equation*}
    Q^-_{r,\delta}(x_0,t_0)=B(x_0,r)\times(t_0-\delta r^p,t_0).   
\end{equation*}
We also denote $Q^-_{r}(x_0,t_0)=Q^-_{r,1}(x_0,t_0)$.
\label{section: VD + PI implies PHI}
\subsection{Local boundedness}
We begin by showing that subsolutions are locally bounded, by using the De Giorgi method. 
\begin{proposition}
\label{prop: loc bdd of subsol}
Assume that $\mu$ is doubling and that a $(1,p)$-Poincar\'e inequality holds in $(X,d,\mu)$.
Let $u$ be a subsolution of \eqref{eq: PDE} in $\Omega_T$. Then there exists a constant c, depending on $c_\mu$, $c_P$ and $p$, such that
\begin{equation*}
    \esssup_{Q^-_{\frac r2}(x_0,t_0)} u \leq c\biggl(  \dashint_{Q^-_r(x_0,t_0)} u_+^p\dnu\biggr)^\frac{1}{p}
\end{equation*}
for every $Q_r^-(x_0,t_0)\subset\Omega_T$ with $0<r\leq\frac13\diam(X)$.
\end{proposition}
\begin{proof}
For $n\in \mathbb N$ and $k>0$, let
\begin{equation*}
k_n = (1-2^{-n})k, \quad \Tilde k_n=\frac{k_{n+1}+k_n}{2}, 
\quad r_n = \frac{r}{2} + \frac{r}{2^{n+1}},
\end{equation*}
$B_n =B(x_0,r_n)$ and $Q_n= Q^-_{r_n}(x_0,t_0)$.
Let $\eta_n$ be a $\frac{c}{r_{n}-r_{n+1}}$-Lipschitz cutoff function with $0\leq \eta\leq 1$, $\eta=1$ on $B_{n+1}$ and $\supp \eta \subset B_n$. Let
\begin{equation*}
    \xi_n(t)
        =\begin{cases}
            0 ,&\quad t\leq t_0-r_{n}^p,\\
            \frac{1}{r_n^p-r_{n+1}^p}(t-t_0+r_{n}^p), & \quad t_0-r_n^p <t\leq t_0-r_{n+1}^p, \\
            1,&\quad t\geq t_0-r_{n+1}^p,
        \end{cases}
\end{equation*}
and denote $\varphi_n=\eta_n\xi_n$. By \Cref{lemma: par sob} we have
\begin{equation}
\label{eq: loc bdd from par sob}
\begin{split}
    &\dashint_{Q_{n+1}}(u-k_{n+1})_+^p\dnu\\
    &\qquad= (k_{n+1}-\Tilde k_n)^{p(1-\kappa)} \dashint_{Q_{n+1}}(u-k_{n+1})_+^p(k_{n+1}-\Tilde k_n)^{p(\kappa-1)}\dnu \\
    &\qquad\leq ck^{1-\kappa}2^{np(\kappa-1)}\dashint_{Q_n}((u-\Tilde k_n)_+\varphi_n)^{p\kappa}\dnu \\
    &\qquad\leq ck^{1-\kappa}2^{np(\kappa-1)} r_n^p \dashint_{Q_n}g_{(u-\Tilde k_n)_+\varphi_n}^p\dnu
     \esssup_{t_0-r_n^p<t<t_0}\biggl(\dashint_{B_n}(u(t)-\Tilde k_n)_+^p \varphi_n^p\dmu\biggr)^{\kappa-1},
\end{split}
\end{equation}
where $\kappa>1$ is as in \Cref{lemma: par sob}.

We claim that 
\begin{equation}
\label{eq: wanted est for loc bdd from cacciop}
    \begin{split}
     &\esssup_{t_0-r_n^p<t<t_0}\dashint_{B_n} (u(t)-\Tilde k_n)_+^p \varphi_n^p\dmu + r_n^p\dashint_{Q_n}g_{(u-\Tilde k_n)_+ \varphi_n }^p\dnu\\
     &\qquad\leq c  2^{(2p+2)n}\dashint_{Q_n}(u-k_n)^p\dnu.
     \end{split}
\end{equation}
Let $h>0$. Using \Cref{lemma: caccioppoli with cutoff} we have
\begin{equation}
\label{eq: general est for loc bdd from cacciop}
\begin{split}
    & \esssup_{t_0-r_n^p<t<t_0}\dashint_{B_n} (u(t) +h)^{p-2}(u(t)-h)_+^2 \varphi_n^p\dmu + r_n^p\dashint_{Q_n}g_{(u-h)_+}^p\varphi_n^p\dnu \\
    &\qquad \leq c 2^{pn}\dashint_{Q_n}(u-h)_+^p +( u_+ +h )^{p-2}(u-h)_+^2\dnu.
\end{split}
\end{equation}
We consider two ranges of $p$ separately. First, assume $1<p<2$. Then
\begin{equation*}
\begin{split}
    (u-\Tilde k_n)_+^p & \leq (u+k_n)^p \chi_{u>\Tilde k_n} 
     =(u+k_n)^{p-2} ((u-k_n)_+ + k_n)^2\chi_{u>\Tilde k_n} \\
    &\leq c  4^{n} (u+k_n)^{p-2} ((u-k_n)_+ + (\Tilde k_n-k_n))^2\chi_{u>\Tilde k_n} \\
    &\leq c  4^{n} (u+k_n)^{p-2} (u-k_n)_+^2
    \leq c 4^n (u-k_n)_+^p.
\end{split}
\end{equation*}
Thus, the product rule and \eqref{eq: general est for loc bdd from cacciop} with $h=k_n$ gives
\begin{equation*}
\begin{split}
     &\esssup_{t_0-r_n^p<t<t_0}\dashint_{B_n} (u(t)-\Tilde k_n)_+^p \varphi_n^p\dmu + r_n^p\dashint_{Q_n}g_{(u-\Tilde k_n)_+\varphi_n}^p\dnu \\
     &\qquad\leq c 2^{2n}\esssup_{t_0-r_n^p<t<t_0}\dashint_{B_n} (u(t)+k_n)^{p-2}(u(t)-k_n)_+^2 \varphi_n^p\dmu\\
     &\qquad\qquad+ c r_n^p \dashint_{Q_n}g_{(u-k_n)_+}^p\varphi_n^p\dnu
      + c r_n^p\dashint_{Q_n} (u-\Tilde k_n)_+^p g_{\varphi_n}^p\dnu \\
     &\qquad \leq c 2^{(p+2)n}\dashint_{Q_n}(u-k_n)_+^p\dnu,
\end{split}
\end{equation*}
which shows \eqref{eq: wanted est for loc bdd from cacciop}. 
Then assume $2\leq p<\infty$. In this case
\begin{equation*}
\begin{split}
    (u-\Tilde k_n)_+^p &\leq (u+\Tilde k_n)^{p-2}(u-\Tilde k_n)_+^2 \leq c 2^{(p-2)n}(u-k_n)_+^p.
\end{split}
\end{equation*}
Applying \eqref{eq: general est for loc bdd from cacciop} with $h=\Tilde k_n$ implies 
\begin{equation*}
\begin{split}
     &\esssup_{t_0-r_n^p<t<t_0}\dashint_{B_n} (u(t)-\Tilde k_n)_+^p \varphi_n^p\dmu + r_n^p\dashint_{Q_n}g_{(u-\Tilde k_n)_+\varphi_n^p}^p\dnu \\
     &\qquad \leq \esssup_{t_0-r_n^p<t<t_0}\dashint_{B_n} (u(t)+\Tilde k_n)^{p-2}(u(t)-\Tilde k_n)_+^2 \varphi_n^p\dmu \\
     &\qquad\qquad+ cr_n^p\dashint_{Q_n}g_{(u-\Tilde k_n)_+}^p\varphi_n^p\dnu
     + cr_n^p\dashint_{Q_n}(u-\Tilde k_n)_+^pg_{\varphi_n}^p\dnu \\
     &\qquad\leq c 2^{(2p-2)n}\dashint_{Q_n}(u-k_n)_+^p\dnu,
\end{split}
\end{equation*}
and hence \eqref{eq: wanted est for loc bdd from cacciop} holds in this case as well. 

We apply \eqref{eq: wanted est for loc bdd from cacciop} in \eqref{eq: loc bdd from par sob} to conclude 
\begin{equation*}
    \dashint_{Q_{n+1}}(u-k_{n+1})_+^p\dnu \leq ck^{p(\kappa-1)}2^{(p(\kappa-1)+\kappa(2p+2))n}\biggl(\dashint_{Q_n} (u-k_n)_+^p\dnu \biggr)^{\kappa}.
\end{equation*}
Denoting 
\begin{equation*}
    Y_n= \biggl(\dashint_{Q_n}(u-k_n)_+^p\dnu\biggr)^\frac{1}{p} 
    \quad\text{and}\quad 
    b=2^{(\kappa-1)+\kappa(2+\frac2p)},
\end{equation*}
we have 
\begin{equation*}
    Y_{n+1}\leq ck^{1-\kappa}b^n Y_n^{\kappa}
\end{equation*}
for every $n\in\mathbb N$.
Applying \cite[Lemma 5.1]{DiBenedettoGIanazzaVespri2012} we conclude that $Y_n\to 0$ as $n\to \infty$ if 
\begin{equation*}
    Y_0 \leq c^{1-\kappa} b^{-(\kappa-1)^2} k.
\end{equation*}
We have
\begin{equation*}
    \dashint_{Q^-_{\frac r2}(x_0,t_0)} (u-k)_+ \dnu \leq c\dashint_{Q^-_{r_n}(x_0,t_0)} (u-k)_+ \dnu \leq cY_n,
\end{equation*}
for every $n\in\mathbb N$. Hence, for $k=cY_0$, with $c$ large enough, we have 
\begin{equation*}
    \dashint_{Q^-_{\frac r2}(x_0,t_0)} (u-k)_+ \dnu = 0.
\end{equation*}
It follows that 
\begin{equation*}
    \esssup_{Q_{\frac r2}^-(x_0,t_0)} u \leq k = c \biggl(\dashint_{Q_r^{-}(x_0,t_0)}u_+^p\dnu\biggr)^\frac{1}{p}.
\end{equation*}
\end{proof}

\subsection{Propagation of positivity in measure}
The following lemma shows that the density of upper level sets propagates in time. 
\begin{lemma}
\label{lemma: propagation of pos in meas}
Assume that $\mu$ is doubling and that a $(1,p)$-Poincar\'e inequality holds in $(X,d,\mu)$. Let $u$ be a nonnegative supersolution in $\Omega_T$, $k>0$, $0<\alpha<1$ and $\sigma >1$. Then there exist constants $0<\delta_0<1$ and $0<\epsilon<1$, depending on $c_\mu$, $p$, $\alpha$ and $\sigma$, such that if $B(x_0,\sigma r)\times[t_0,t_0+\delta r^p)\subset\Omega_T$, $0<\delta\leq \delta_0$, and $t_0$ is a Lebesgue point of $u$ with 
\begin{equation*}
    \mu(\{u(t_0)\geq k\}\cap B(x_0,r)) \geq \alpha \mu(B(x_0,r)),
\end{equation*}
then 
\begin{equation*}
    \mu( \{u(t)\geq \epsilon k\}\cap B(x_0,r))) \geq \frac{\alpha}{2} \mu(B(x_0,r))
\end{equation*}
for almost every $t_0<t<t_0+\delta r^p$.
\end{lemma}
\begin{proof}
Let $\eta>0$ with $1+\eta<\sigma$. Applying \Cref{lemma: caccioppoli with cutoff} we have
\begin{equation}
\label{eq: first est in prop of pos}
\begin{split}
    &\int_{B(x_0,r)} G_{-}(u(t),k)\dmu 
    \leq \int_{B(x_0,{(1+\eta)r})}G_{-}(u(t_0),k)\dmu\\
    &\qquad+ \frac{c}{\eta^p r^p}\int_{B(x_0,(1+\eta)r)\times (t_0, t_0 + \delta r^p)} (u-k)_-^p\dnu.
\end{split}
\end{equation}
To estimate the first term on the right-hand side of the above, we split the integral and obtain
\begin{equation*}
\begin{split}
    & \int_{B(x_0,{(1+\eta)r})}G_{-}(u(t_0),k)\dmu 
     \leq (p-1) k^p \mu(B(x_0,(1+\eta)r)\backslash \mu(B(x_0,r))) \\
    & \qquad\qquad + (p-1)\mu(\{u(t_0)<k\}\cap B(x_0,r))\int_{0}^{k} s^{p-2}(k-s)\ds \\
    & \qquad\leq c k^p \eta^\beta \mu(B(x_0,r)) 
    + (p-1)(1-\alpha)\mu(B(x_0,r))\int_{0}^{k} s^{p-2}(k-s)\ds .
\end{split}
\end{equation*}
Here, we used \eqref{eq: geodesic doubling becomes annular decay} in the second inequality. We note that
\begin{equation*}
    (p-1)\int_{0}^{k} s^{p-2}(k-s)\ds \leq (p-1) \int_{\epsilon k}^{k} s^{p-2}(k-s)\ds + c(\epsilon k)^p.
\end{equation*}
To estimate the second term on the right-hand side of \eqref{eq: first est in prop of pos} we have
\begin{equation*}
    \frac{c}{\eta^p r^p}\int_{B(x_0,(1+\eta)r)\times (t_0, t_0 + \delta r^p)} (u-k)_-^p\dnu \leq \frac{c \delta_0}{\eta^p r^p} k^p \mu(B(x_0,r)).
\end{equation*}
Finally, we estimate the left hand side of \eqref{eq: first est in prop of pos} from below with 
\begin{equation*}
    \int_{B(x_0,r)} G_{-}(u(t),k)\dmu 
    \geq (p-1) \mu(\{u(t)<\epsilon k\}\cap B(x_0,r)))\int_{\epsilon k}^k s^{p-2}(k-s)\ds 
\end{equation*}
Combining the estimates we have obtained thus far we have
\begin{equation*}
    \mu(\{u(t)<\epsilon k\}\cap B(x_0,r))) 
    \leq \mu(B(x_0,r)) \Bigl( (1-\alpha) + c\epsilon^p+ c \eta^\beta + \frac{c \delta_0}{\eta^p} \Bigr).
\end{equation*}
Choosing $\epsilon$, $\eta$, and finally $\delta_0$ to be small enough, the previous estimate implies 
\begin{equation*}
    \mu(\{u(t)<\epsilon k\}\cap B(x_0,r))) \leq \Bigl(1-\frac{\alpha}{2}\Bigr)\mu(B(x_0,r)).
\end{equation*}
\end{proof}
Next, we show that the density of upper level sets can be improved.
\begin{lemma}
\label{lemma: shrinking lemma short times}
Assume that $\mu$ is doubling and that a $(1,p)$-Poincar\'e inequality holds in $(X,d,\mu)$.  Let $u$ be a nonnegative supersolution in $\Omega_T$, $k>0$, $0<\alpha<1$, and $\sigma>1$. Let $\delta_0$ and $\epsilon$ be as in \Cref{lemma: propagation of pos in meas} and let $0<\delta\leq \delta_0$. Then there exist constants $c$ and $\theta$, depending on $\alpha$, $\sigma$, $p$, $c_\mu$, $c_P$ and $\delta$, such that if $B(x_0,\sigma r)\times[t_0,t_0+\delta r^p)\subset \Omega_T$ and $t_0$ is a Lebesgue point of $u$ with 
\begin{equation*}
    \mu(\{u(t_0) > k \}\cap B(x_0,r))) \geq \alpha \mu(B(x_0,r)),
\end{equation*}
then
\begin{equation*}
    \nu\Bigl(\Bigl\{ u < \frac{\epsilon k}{2^j}\Bigr\}\cap Q\Bigr) \leq \frac{c}{j^\theta} \nu(Q), 
\end{equation*}
for every $j\in \mathbb N$, where  $Q=B(x_0,r)\times (t_0,t_0+\delta r^p)$.
\end{lemma}

\begin{proof}
Let $k_i=\frac{\epsilon}{2^i}k$, where $i\in \{0,\dots,j\}$. Denote $v=\max\{u,k_i\}-\max\{u,k_{i+1}\}$. Applying \Cref{lemma: propagation of pos in meas} we find that 
\begin{equation*}
    \mu(\{v(t) > 0\}\cap B(x_0,r)) ) 
    \leq \Big(1-\frac{\alpha}{2}\Big)\mu(B(x_0,r)),
\end{equation*}
for almost every $t_0<t<t_0+\delta r^p$. For such $t$, we apply \Cref{eq: Poincare for function zero on large set} with exponent $1\leq q<p$ on the right-hand side and obtain
\begin{equation*}
\begin{split}
    &(k_i-k_{i+1})\mu(\{u(t)<k_{i+1}\}\cap B(x_0,r))) 
     \leq \int_{B(x_0,r)} v(t)\dmu \\
    &\qquad \leq c r \mu(B(x_0,r))^{1-\frac{1}{q}} \biggl(\int_{B(x_0,r)}g_{v(t)}^q \dmu \biggr)^\frac{1}{q}.
\end{split}
\end{equation*}
Integrating the estimate above over $(t_0,t_0+\delta r^p)$ and using H\"older's inequality we obtain
\begin{equation*}
\begin{split}
    &(k_i-k_{i+1}) \nu(\{u<k_{i+1}\}  \cap Q) \\
     &\qquad \leq c r \mu(B(x_0,r))^{1-\frac{1}{q}} \int_{t_0}^{t_0+\delta r^p}\biggl(\int_{B(x_0, r)}g_{v(t)}^q \dmu \biggr)^\frac{1}{q} \dt \\
    &\qquad \leq cr \nu(Q)^{1-\frac{1}{q}} \biggl( \int_{Q} g_v^q \dnu \biggr)^\frac{1}{q}\\
    &\qquad \leq cr \nu(Q)^{1-\frac{1}{q}} \biggl( \int_{Q} g_{(u-k_{i})_-}^p \dnu \biggr)^\frac{1}{p} \nu(\{ k_{i+1}<u<k_i\}\cap Q )^{\frac{1}{q}-\frac{1}{p}}.
\end{split}
\end{equation*}
Here, we used the fact that $g_v=g_u \chi_{\{k_{i+1} < u <k_i \}}$. 
By \Cref{lemma: caccioppoli with cutoff} we have
\begin{equation*}
\begin{split}
    \int_{Q} g_{(u-k_i)_-}^p \dnu 
    & \leq \frac{c}{r^p}\int_{B(x_0,\sigma r)\times (t_0,t_0+\delta r^p)} (u-k_i)^p_-\dnu\\
    &\quad+ c\int_{B(x_0,\sigma r)} G_-(u(t_0),k_i)\dmu 
     \leq c r^{-p} \nu(Q) k_i^p,
\end{split}
\end{equation*}
which implies 
\begin{equation*}
\begin{split}
    (k_i-k_{i+1})  \nu(\{u<k_{j+1}\} \cap Q) \leq k_i c \nu(Q)^{1-\frac{1}{q}+\frac{1}{p}}  \nu(\{ k_{i+1}<u<k_i\}\cap Q )^{\frac{1}{q}-\frac{1}{p}}.
\end{split}
\end{equation*}
We conclude that 
\begin{equation*}
    \nu(\{u<k_{i+1}\} \cap Q)^{\frac{pq}{p-q}} \leq  c \nu(Q)^{\frac{pq}{p-q}-1}  \nu(\{ k_{i+1}<u<k_i\}\cap Q).
\end{equation*}
Summing the previous inequality over $i\in \{0,\dots,j\}$, we obtain
\begin{equation*}
    j \nu(\{u<k_{j}\} \cap Q)^{\frac{pq}{p-q}} \leq  c \nu(Q)^{\frac{pq}{p-q}},
\end{equation*}
and hence 
\begin{equation*}
    \nu(\{u<k_{j}\} \cap Q) \leq  c j^{-\frac{p-q}{pq}} \nu(Q),
\end{equation*}
which proves the claim for $\theta=\frac{p-q}{pq}$.
\end{proof}
The following corollary removes the restriction $\delta\leq \delta_0$ from \Cref{lemma: shrinking lemma short times}.
\begin{corollary}
\label{cor: shrinking lemma any time}
Assume that $\mu$ is doubling and that a $(1,p)$-Poincar\'e inequality holds in $(X,d,\mu)$. Let $u$ be a nonnegative supersolution in $\Omega_T$, $k>0$, $0<\alpha<1$, $0<\gamma<1$, $\delta>0$ and $\sigma >1$. Then there exists a constant $\lambda>0$, depending on $p$, $c_\mu$, $c_P$, $\alpha$, $\gamma$, $\delta$ and $\sigma$, such that if $B(x_0,\sigma r)\times[t_0,t_0+\delta r^p)\subset\Omega_T$ and $t_0$ is a Lebesgue point of $u$ with 
\begin{equation*}
    \mu(\{u(t_0)>k\}\cap B(x_0,r))) \geq \alpha \mu(B(x_0,r)),
\end{equation*}
then
\begin{equation*}
    \nu\big(\big\{ u < \lambda k\big\}\cap Q\big) \leq \gamma \nu(Q) ,
\end{equation*}
where $Q=B(x_0,r)\times (t_0,t_0+\delta r^p)$.
\end{corollary}
\begin{proof}
The claim follows directly from \Cref{lemma: shrinking lemma short times} if $\delta \leq \delta_0$, where $\delta_0$ is as in \Cref{lemma: propagation of pos in meas}. Thus, we may assume that $\delta > \delta_0$. Let $k\in\mathbb N$ be such that $k\frac{\delta_0}{2}<\delta\leq (k+1)\frac{\delta_0}{2}$. Denote 
\begin{equation*}
    Q_i = B(x_0,r)\times (t_0+i\tfrac{\delta_0}{2} r^p, t_0+(i+2)\tfrac{\delta_0}{2} r^p)
\end{equation*}
for $i\in \{0,\dots,k-2\}$ and set $Q_{k-1}=B(x_0,r)\times (t_0+k\frac{\delta_0}{2} r^p,t_0+\delta r^p)$. Applying \Cref{lemma: shrinking lemma short times} to $Q_0$ we find that 
\begin{equation*}
    \nu\big(\bigl\{ u < \lambda k\big\}\cap Q_0\bigr) 
    \leq\frac{\gamma}{2}\nu(Q_0), 
\end{equation*}
for $\lambda$ depending only on $p$, $c_\mu$, $c_P$, and $\alpha$. Hence, there exists $t_0+\frac{\delta_0}{2} r^p<t_1< t_0+\delta_0 r^p$ such that
\begin{equation*}
    \mu(\{u(t_1)<k\}\cap B(x_0,r))) \leq \nu \mu(B(x_0,r)).
\end{equation*}
Applying \Cref{lemma: shrinking lemma short times} to $Q_1$ we conclude that
\begin{equation*}
    \nu(\{ u < \lambda^2 k\}\cap Q_1) 
    \leq\frac{\gamma}{2}\nu(Q_1).
\end{equation*}
Continuing this process recursively, we find $\lambda$ depending on $p$, $c_\mu$, $c_P$, $\alpha$ and $\delta$ such that 
\begin{equation*}
    \nu(\{ u < \lambda k\}\cap Q_i) 
    \leq\frac{\gamma}{2}\nu(Q_i),
\end{equation*}
for every $i\in \{0,\dots,k-1\}$. For such $\lambda$ we deduce that 
\begin{equation*}
\begin{split}
    \nu\big(\big\{ u < \lambda k\big\}\cap Q\big)
    & \leq \sum_{i=0}^{k-1} \nu\big(\big\{ u < \lambda k\big\}\cap Q_i\big) 
     \leq \sum_{i=0}^{k-1}\frac{\gamma}{2}\nu( Q_i)
     \leq \gamma \nu(Q),
\end{split}
\end{equation*}
which finishes the proof.
\end{proof}

\subsection{A De Giorgi--DiBenedetto lemma}
The final step to showing the expansion of positivity is the following De Giorgi--DiBenedetto lemma, which is proved in a similar way as \Cref{prop: loc bdd of subsol}.
\begin{lemma}
\label{lemma: de giorgi lemma}
Assume that $\mu$ is doubling and that a $(1,p)$-Poincar\'e inequality holds in $(X,d,\mu)$.
Let $\delta>0$, $\sigma>1$, $\mu^-\in \mathbb R$. Assume that $u$ is a supersolution of \eqref{eq: PDE} in $Q^-_{\sigma r,\delta}(x_0,t_0)$, where $0<\sigma r\leq\frac13\diam(X)$ and $u \geq \mu^-$ almost everywhere in  $Q_{\sigma r,\delta}^-(x_0,t_0)$.
Let $0<a<1$ and $M>1$. There exists a constant $0<\gamma<1$, depending on $c_\mu, c_P$, $p$, $\delta$, $\sigma$, $a$ and  $M$, such that if 
\begin{equation}
    \label{eq: meas dens cond in de giorgi lemma}
    \nu(\{ u -\mu^-< k  \}\cap Q_{\sigma r,\delta}^-(x_0,t_0)) \leq \gamma \nu(Q_{\sigma r,\delta}^-(x_0,t_0))
\end{equation}
and $ \lvert \mu^-\rvert\leq M k$, then $u -\mu^-\geq a k$ $\nu$-almost everywhere in $Q_{r,\delta}^-(x_0,t_0)$.
\end{lemma}
\begin{proof}
For $n\in \mathbb N$, let
\begin{equation*}
        k_n = \mu^-+ak+\frac{1-a}{2^n}k, \quad \Tilde k_n=\frac{k_{n+1}+k_n}{2}, \quad r_n = r + \frac{\sigma-1}{2^{n}} r,
\end{equation*}
$B_n =B(x_0,r_n)$ and $Q_n= B_n \times (t_0-\delta r_n^p,t_0)$.
Let $\eta_n$ be a $\frac{c}{r_n-r_{n+1}}$-Lipschitz cutoff function with $0\leq \eta\leq 1$, $\eta=1$ on $B_{n+1}$ and $\supp \eta \subset B_n$. Let
\begin{equation*}
    \xi_n(t)
        =\begin{cases}
            0 ,&\quad t\leq t_0-\delta r_{n}^p,\\
            \frac{t-t_0+\delta r_{n}^p}{\delta(r_n^p-r_{n+1}^p)}, & \quad t_0-\delta r_n^p <t\leq t_0-\delta r_{n+1}^p, \\
            1,&\quad t\geq t_0-\delta r_{n+1}^p,
        \end{cases}
\end{equation*}
and denote $\varphi_n=\eta_n\xi_n$. Let $\kappa$ be as in \Cref{lemma: par sob}.  \Cref{lemma: par sob} gives
\begin{equation}
\label{eq: de giorgi lemma est from par sob}
\begin{split}
    k^{p\kappa} \frac{\nu(A_{n+1})}{\nu(Q_{n+1})}  
    &\leq c2^{n(p\kappa)}\dashint_{Q_n}((u-\Tilde k_n)_-\varphi_n)^{p\kappa}\dnu\\
    &\leq c2^{n(p\kappa)}r_n^p\dashint_{Q_n}g_{(u-\Tilde k_n)_-\varphi_n}^p\dnu \esssup_{t_0-r_n^p<t<t_0}\biggl(\dashint_{B_n}(u(t)-\Tilde k_n)_-^p \varphi_n^p\dmu\biggr)^{\kappa-1}.
\end{split}
\end{equation}
Let us estimate the right hand side of \eqref{eq: de giorgi lemma est from par sob}. Using \Cref{lemma: caccioppoli with cutoff}, we have 
\begin{equation}
\label{eq: cacciop for de giorgi lemma}
\begin{split}
     &\esssup_{t_0-r_n^p<t<t_0} \dashint_{B_n} (\lvert u(t) \rvert +\lvert k_n\rvert )^{p-2}(u(t)-k_n)_-^2 \varphi_n^p\dmu + r_n^p \dashint_{Q_n}g_{(u-k_n)_-}^p \varphi_n^p\dnu \\
    &\qquad \leq c 2^{pn}\dashint_{Q_n}(u-k_n)_-^p +( \lvert u\rvert +\lvert k_n\rvert )^{p-2}(u-k_n)_-^2\dnu \\
    &\qquad \leq c  2^{pn} k^p \frac{\nu(A_n)}{\nu(Q_n)},
\end{split}
\end{equation}
where in the last step we used the assumption $\lvert \mu^-\rvert\leq M k$. Let us compare $(u-\Tilde k_n)_-^p$ and $(\lvert u \rvert +\lvert k_n\rvert )^{p-2}(u-k_n)_-^2$. Assume first that $1<p<2$. Since $\lvert \mu^-\rvert\leq M k$, we have $\lvert u\rvert + \lvert k_n\rvert \leq c k$ on $\{u<\Tilde k_n\}$, and hence 
\begin{align*}
    (u-\Tilde k_n)_-^p\leq k^p\chi_{u<\Tilde k_n} 
    &\leq c (\lvert u\rvert +\lvert k_n\rvert )^{p-2} k^2\chi_{u<\Tilde k_n}\\ 
    &\leq c 2^{2n}c (\lvert u\rvert +\lvert k_n\rvert )^{p-2} (u-k_n)_-^2.
\end{align*}
For $2\leq p <\infty$, we simply estimate
\begin{equation*}
    (u-\Tilde k_n)_-^p\leq (u-k_n)_-^p\leq  (\lvert u\rvert +\lvert k_n\rvert )^{p-2} (u-k_n)_-^2.
\end{equation*}
In either case we conclude that 
\begin{align*}
    &\esssup_{t_0-r_n^p<t<t_0}\dashint_{B_n} (u(t)-\Tilde k_n)_+^p \varphi_n^p\dmu\\
    &\qquad\leq c2^{2n}\esssup_{t_0-r_n^p<t<t_0} \dashint_{B_n} (\lvert u(t) \rvert +\lvert k_n\rvert )^{p-2}(u(t)-k_n)_-^2 \varphi_n^p\dmu.
\end{align*}
Using the previous estimate, \eqref{eq: cacciop for de giorgi lemma} and the product rule we obtain
\begin{equation}
\label{eq: wanted est for de giorgi lemma}
\begin{split}
    & \esssup_{t_0-r_n^p<t<t_0}\dashint_{B_n} (u(t)-\Tilde k_n)_-^p \varphi_n^p\dmu + r_n^p\dashint_{Q_n}g_{(u-\Tilde k_n)_- \varphi_n}^p\dnu  \\
    &\qquad \leq  c2^{2n}\esssup_{t_0-r_n^p<t<t_0} \dashint_{B_n} (\lvert u(t) \rvert +\lvert k_n\rvert )^{p-2}(u(t)-k_n)_-^2 \varphi_n^p\dmu \\
    &\qquad\qquad+ c r_n^p\dashint_{Q_n}g_{(u-k_n)_-}^p \varphi_n^p\dnu 
     + cr_n^p\dashint_{Q_n}(u-k_n)_-^p g_{\varphi_n}^p\dnu\\
    & \qquad\leq c 2^{(p+2)n} k^p \frac{\nu(A_n)}{\nu(Q_n)}.
\end{split}
\end{equation}
Applying \eqref{eq: wanted est for de giorgi lemma} in \eqref{eq: de giorgi lemma est from par sob} we find that 
\begin{equation}
    \label{eq: decay in de giorgi lemma proof}
    \frac{\nu(A_{n+1})}{\nu(Q_{n+1})}\leq c2^{(2p+2)\kappa n}\biggl( \frac{\nu(A_n)}{\nu(Q_n)} \biggr)^{\kappa}.
\end{equation}
Denote $Y_n=\frac{\nu(A_{n+1})}{\nu(Q_{n+1})}$. Applying \cite[Lemma 5.1]{DiBenedettoGIanazzaVespri2012} to \eqref{eq: decay in de giorgi lemma proof} shows that $Y_n\to 0$ as $n\to \infty$ provided that \eqref{eq: meas dens cond in de giorgi lemma} holds with $\gamma$ small enough. Since
\begin{equation*}
    \frac{\nu(\{u<\mu^-+ak\}\cap Q_{r,\delta}^-(x_0,t_0))}{\nu(Q_{ r,\delta}^-(x_0,t_0))} \leq c Y_n,
\end{equation*}
for every $n\in N$, the convergence $Y_n\to 0$ as $n\to\infty$ implies 
\[
\essinf_{Q_{r,\delta}^-(x_0,t_0)} u \geq \mu^-+ak.
\]
\end{proof}

\begin{remark}
\label{remark: resctrion on radius in de giorgi lemma can be removed}
The restriction $0<\sigma r\leq\frac13\diam(X)$ in \Cref{lemma: de giorgi lemma} can be removed. To see this, let $\delta>0$, $\sigma>1$, $\mu^-\in \mathbb R$, and assume that $u$ is a supersolution in $Q^-_{\sigma r,\delta}(x_0,t_0)$ and
$u \geq \mu^-$ almost everywhere in $Q_{\sigma r,\delta}^-(x_0,t_0)$.
Let $0<a<1$ and $M>1$. We may assume $\sigma \leq 2$. If $r\leq\frac16\diam(X)$, we simply apply \Cref{lemma: de giorgi lemma}. Thus, we may assume that $\frac16\diam(X)\leq r \leq \diam(X)$. The doubling condition implies that there exists $\theta$ depending only on $c_\mu$, $p$ and $\sigma$, such that if 
\begin{equation*}
    \nu(\{ u -\mu^-< k  \}\cap Q_{\sigma r,\delta}^-(x_0,t_0)) \leq \theta \gamma \nu(Q_{\sigma r,\delta}^-(x_0,t_0))
\end{equation*}
then 
\begin{equation}
    \label{eq: meas dens cond for applying de giorgi in smaller cylinder}
    \nu(\{ u -\mu^-< k  \}\cap Q_{(\sigma -1) r,\delta}^-(z)) \leq  \gamma \nu(Q_{(\sigma-1)r,\delta}^-(z))
\end{equation}
for every $z\in Q_{r,\delta}(x_0,t_0)$. We use \eqref{eq: meas dens cond for applying de giorgi in smaller cylinder} and apply \Cref{lemma: de giorgi lemma} in $Q_{(\sigma-1)r,\delta}^-(z)$ to conclude that $u -\mu^-\geq a k$
almost everywhere in $Q_{(\sigma-1)\frac r2,\delta}^-(z)$. Since $z$ was arbitrary, we conclude $u -\mu^-\geq a k$ almost everywhere in $Q_{r,\delta}^-(x_0,t_0)$.
\end{remark}

Arguing as in \cite{Liao2021}, the following corollary is implied by \Cref{lemma: de giorgi lemma}. We include the proof for the convenience of the reader.
\begin{corollary}
\label{cor: lower semicont of supersolutions}
Assume that $\mu$ is doubling and that a $(1,p)$-Poincar\'e inequality holds in $(X,d,\mu)$.
Let $u$ be a supersolution of \eqref{eq: PDE} in $\Omega_T$. Then the lower semicontinuous function 
\begin{equation*}
    u^*(z)=\lim_{r\to 0} \essinf_{Q_r^-(z)} u
\end{equation*}
agrees with $u$ $\nu$-almost everywhere in $\Omega_T$.
\end{corollary}
\begin{proof}
By \Cref{prop: loc bdd of subsol}, a supersolution is locally bounded from below. Let $z\in\Omega_T$ be a Lebesgue point of $u$. Clearly $u^*(z)\leq u(z)$. Suppose that the reverse inequality does not hold, so that $u^*(z)<u(z)$. Let $R>0$ be such that $Q^-_{R}(z)\Subset\Omega_T$ and $\mu^-=\essinf_{Q^-_{R}(z)} u$. Since $\mu^-\leq u^*(z)<u(z)$ there exist $k>0$ and $0<a<1$ such that
\begin{equation*}
    \mu^- + k < u(z) 
    \quad\text{and}\quad 
    \mu^-+ a k > u^*(z).
\end{equation*}
Denote $M=\frac{\lvert \mu^-\rvert}{k}$. Let $\gamma$ be as in \Cref{lemma: de giorgi lemma} corresponding to the parameters $\delta=1$, $\sigma=2$, $a$ and $M$. We claim that 
\begin{equation}
    \label{eq: density for applying de giorgi lemma for lower semicont}
    \nu(\{u<\mu^-+M\}\cap Q^-_{\rho}(z))<\nu \mu(Q^-_\rho(z)),
\end{equation}
for some $0<\rho<R$. If \eqref{eq: density for applying de giorgi lemma for lower semicont} does not hold, then 
\begin{equation*}
\begin{split}
    \int_{Q^-_{\rho}(z)} \lvert u(z)-u\rvert \dnu 
    & \geq (u(z)-(\mu^-+k) \nu(\{u< \mu^-+M\}\cap Q^-_{\rho}(z)) \\
    & \geq (u(z)-(\mu^-+k))\nu(Q^-_\rho(z)),
\end{split}
\end{equation*}
for every $0<\rho<R$, which contradicts the assumption that $z$ is a Lebesgue point. Hence, since \eqref{eq: density for applying de giorgi lemma for lower semicont} holds we apply \Cref{lemma: de giorgi lemma} and conclude that 
\begin{equation*}
    u \geq \mu^-+ ak>u^*(z) 
\end{equation*}
$\nu$-almost everywhere in  $Q^-_{\frac{\rho}{2}}(z)$.
This contradicts the definition of $u^*(z)$, and thus we conclude that $u(z)\leq u^*(z)$.
\end{proof}

\subsection{Expansion of positivity}
We are ready to combine the previous results to conclude the expansion of positivity property, which will play an essential role in the proof of the parabolic Harnack inequality.
\begin{proposition}
\label{prop: expansion of pos}
Assume that $\mu$ is doubling and that a $(1,p)$-Poincar\'e inequality holds in $(X,d,\mu)$. Let $u$ be a nonnegative supersolution in $\Omega_T$, $k>0$, $0<\alpha<1$, $\delta>0$ and $\sigma>1$. Then there exists a constant $\lambda$, depending on $c_\mu$, $c_P$, $p$, $\alpha$, $\delta$ and $\sigma$, such that if $B(x_0,\sigma r)\times[t_0,t_0+\delta r^p)\subset\Omega_T$ and $t_0$ is a Lebesgue point of $u$ with  
\begin{equation}
    \label{eq: assumption in exp of pos}
    \mu( \{u(t_0)<k\}\cap B(x_0,r)) \leq \alpha \mu(B(x_0,r)),
\end{equation}
then  $u \geq \lambda k$ almost everywhere in $B(x_0,r)\times(t_0+2^{-p}\delta r^p,t_0+\delta r^p)$.
\end{proposition}
\begin{proof}
Let $h>0$ and let $1<\tau \leq \min\{\sqrt{\sigma},2\}$ with $\frac{1}{\tau^p}\geq (1-2^{-p})$. Applying \Cref{lemma: de giorgi lemma} with the parameters $\frac{\delta}{\tau^p}$, $\tau$ and $a=\frac{1}{2}$, there exists a constant $\gamma$, depending on $c_\mu$, $c_P$, $p$, $\alpha$ and $\delta$, such that if 
\begin{equation*}
    \nu(\{ u < h \}\cap Q_{\tau r,\frac{\delta}{\tau^p}}^-(x_0,t_0+\delta r^p))\leq \gamma \nu(Q_{\tau r,\frac{\delta}{\tau^p}}^-(x_0,t_0+\delta r^p)),
\end{equation*}
then $u \geq \frac{1}{2}h$ $\nu$-almost everywhere in $B(x_0,r)\times(t_0+2^{-p}\delta r^p,t_0+\delta r^p)$.
On the other hand \eqref{eq: assumption in exp of pos} implies that
\begin{equation*}
    \mu(\{u(t_0)<k\}\cap B(x_0,\tau r))\leq \Bigl(1-\frac{1-\alpha}{c_\mu}\Bigr)\mu(B(x_0,\tau r)),
\end{equation*}
and hence \Cref{cor: shrinking lemma any time} shows that 
\begin{equation*}
    \nu(\{ u < \lambda k  \}\cap Q_{\tau r,\frac{\delta}{\tau^p}}^-(x_0,t_0+\delta r^p))\leq \gamma \nu(Q_{\tau r,\frac{\delta}{\tau^p}}^-(x_0,t_0+\delta r^p)),
\end{equation*}
for some $\lambda$ as in \Cref{cor: shrinking lemma any time}. 
The claim follows by the beginning of the proof with $h=\lambda k$.
\end{proof}

\begin{corollary}
\label{cor: slightly adjusted exp of pos}
Assume that $\mu$ is doubling and that a $(1,p)$-Poincar\'e inequality holds in $(X,d,\mu)$.
Let $k>0$, $0<\alpha<1$, $\delta>0$ and $\sigma>1$. 
Assume that $u$ is a nonnegative supersolution of \eqref{eq: PDE} in $B(x_0,\sigma r)\times(t_0-\delta r^p,t_0+\delta r^p)$ and that
\begin{equation}
    \label{eq: slightly changed assumption for exp of pos}
    \nu( \{u<k\}\cap Q^-_{r,\delta}(x_0,t_0)) \leq \alpha \nu(Q^-_{r,\delta}(x_0,t_0)).
\end{equation}
Then there exists a constant $\lambda$, depending on $c_\mu$, $c_P$, $p$, $\alpha$, $\delta$ and $\sigma$, such that $u \geq \lambda k$ almost everywhere in $B(x_0,r)\times(t_0+2^{-p}\delta r^p,t_0+\delta r^p)$.    
\end{corollary}
\begin{proof}
By assumption, there exists $t_0-\delta r^p<t<t_0$ such that
\begin{equation*}
    \mu( \{u(t)<k\}\cap B(x_0,r)) \leq \alpha \mu(B(x_0,r)).
\end{equation*}
The claim follows from \Cref{prop: expansion of pos} with a parameter $\delta$ depending on $t$.
\end{proof}
\subsection{Parabolic Harnack inequality}
Finally we apply the expansion of positivity property and the De Giorgi--DiBenedetto lemma to prove the parabolic Harnack inequality.
\begin{theorem}\label{thm: PHI}
Assume that $\mu$ is doubling and that a $(1,p)$-Poincar\'e inequality holds in $(X,d,\mu)$.
Assume that $u$ is a nonnegative solution of \eqref{eq: PDE} in $Q=B(x_0,4r)\times (t_0-(4r)^p,t_0+(4r)^p)$.
Then there exists a constant $c_H>0$, depending on $c_\mu$, $c_P$ and $p$, such that
\begin{equation*}
    \esssup_{Q^-} u \leq c_H \essinf_{Q^+} u,
\end{equation*}
where $Q^-=B(x_0,r)\times (t_0- r^p,t_0-2^{-p}r^p)$
and $Q^+=B(x_0,r)\times (t_0+2^{-p}r^p,t_0+r^p)$.
\end{theorem}
\begin{proof}
We show that $u(z_0)\leq c\essinf_{Q^+} u$
for almost every $z_0\in Q^-$. If $u(z_0)=0$ we are done. Otherwise, we consider $v(z)=\frac{u(z)}{u(z_0)}$.
Let 
\[
m(\rho) = \esssup_{Q_\rho^{-}(z_0)} v
\quad\text{and}\quad 
n(\rho)=\biggl(1-\frac{4\rho}{r}\biggr)^{-\beta}
\]
for $0\le\rho<\frac r4$. By \Cref{cor: lower semicont of supersolutions}, we may assume that $m(0)=1=n(0)$. Let
\begin{equation*}
    r_0 = \sup \{ 0<\rho<r: m(\rho)\geq n(\rho) \}.
\end{equation*}
Note that $r_0<\frac r4$ since $m$ is bounded and $n(\rho)\to\infty$ as $\rho\to \frac r4$. By upper semicontinuity of $m-n$ we have $m(r_0)\geq n(r_0)$. Thus there exists $\Tilde z=(\Tilde x, \Tilde t)\in Q_{r_0}^-(z_0)$ such that 
\begin{equation}
    \label{eq: sup almost realized for exp of pos proced}
    v(\Tilde z)\geq \tfrac{3}{4}m(r_0) \geq \tfrac{3}{4}n(r_0).
\end{equation}
We have $d(x_0,\Tilde x)\leq \frac{5}{4}$. Let $\rho= \frac{\frac r4-r_0}{2}$. Then $Q^-_{\rho}(\Tilde z)\subset Q^-_{\frac{r_0+\frac r4}{2}}(z_0)$ and by definition of $r_0$ we obtain
\begin{equation*}
    \esssup_{Q_{\rho}(\Tilde z)} v \leq \esssup_{Q_{\frac{\frac r4+r_0}{2}}(z)} v \leq n\biggl(\frac{\frac r4+r_0}{2}\biggr)=2^{\beta} n(r_0).
\end{equation*}
We claim that 
\begin{equation}
    \label{eq: measure density dep on beta}
    \nu( \{ v > \tfrac{1}{4}n(r_0) \}\cap Q_{\rho}^-(\Tilde z)) \geq \gamma_0\nu(Q^-_{\rho}(\Tilde z))
\end{equation}
for some $\gamma_0$ depending only on $c_\mu,c_P,p$ and $\beta$. 
If \eqref{eq: measure density dep on beta} does not hold, then we have
\begin{equation*}
    \nu(\{ 2^\beta n(r_0)-v < 2^\beta n(r_0)-\tfrac{1}{4}n(r_0)\}\cap Q^-_{\rho}(\Tilde z)) \leq \gamma_0 \nu(Q_{\rho}^-(\Tilde z)).
\end{equation*}
Applying \Cref{lemma: de giorgi lemma} to the supersolution $-u$, with parameters 
\[
\mu^-=-2^\beta n(r_0), 
\quad a=\frac{2^\beta-\frac{1}{2}}{2^\beta-\frac{1}{4}},\quad\sigma=2
\quad\text{and}\quad M=\frac{2^\beta}{2^\beta-\frac{1}{4}},
\]
we find that 
\begin{equation*}
    2^\beta n(r_0)-v(\Tilde z) 
    \geq a( 2^\beta n(r_0)-\tfrac{1}{4}n(r_0))=2^{\beta} n(r_0)-\tfrac{1}{2}n(r_0),
\end{equation*}
for $\gamma_0>0$ small enough, depending only on $c_\mu$, $c_P$, $p$, and $\beta$. It follows that $v(\Tilde z) \leq \frac{1}{2}n(r_0)$,
which contradicts \eqref{eq: sup almost realized for exp of pos proced}. Thus, we conclude that \eqref{eq: measure density dep on beta} holds. 

This allows us to apply \Cref{cor: slightly adjusted exp of pos} and obtain that 
\begin{equation}
    \label{eq: first expanded pos}
    v \geq \lambda_0 n(r_0) 
    \text{ in } 
    B(\Tilde x, \rho) \times (\Tilde t+2^{-p}\rho^p,\Tilde t+ \rho^p),
\end{equation}
for some $\lambda_0$ depending on $c_\mu, c_P, p$ and $\beta$. Let $0<\delta\leq 1$ to be chosen later, and set $\tau=\sqrt{\frac{10}{9}}$. From \eqref{eq: first expanded pos} and the doubling condition we conclude that
\begin{equation*}
    \nu( \{ v>\lambda_0 n(r_0)\} \cap Q^-_{\tau\rho, \delta}(\Tilde x,\Tilde t + \rho^p) ) \geq \gamma \nu(Q^-_{\tau\rho}(\Tilde x,\Tilde t + \rho^p)), 
\end{equation*}
where $\gamma=\frac{1}{c_\mu 2^{p}}$. Let $\lambda$ be the constant in \Cref{cor: slightly adjusted exp of pos} corresponding to $\alpha=\gamma$, $\delta$ and $\sigma = \tau$. Applying \Cref{cor: slightly adjusted exp of pos} we find that 
\begin{equation*}
    v \geq \lambda \lambda_0 n(r_0) 
    \text{ in } 
    B(\Tilde x, \tau \rho) \times (\Tilde t + \rho^p+\delta 2^{-p}(\tau\rho)^p, \Tilde t +\rho^p+\delta(\tau\rho)^p).
\end{equation*}
This in turn implies 
\begin{equation*}
    \nu( \{ v>\lambda \lambda_0 n(r_0)\} \cap Q^-_{\tau^2\rho, \delta}(\Tilde x,\Tilde t + \rho^p +\delta (\tau\rho)^p) ) \geq \gamma \nu(Q^-_{\tau^2\rho, \delta}(\Tilde x,\Tilde t + \rho^p +\delta (\tau\rho)^p)).
\end{equation*}
Continuing recursively we find that 
\begin{equation}
    \label{eq: iterated expansion of pos}
    v \geq \lambda^n \lambda_0 n(r_0) \text{ in } B(\Tilde x,\tau^n\rho)\times\Bigl(\Tilde t +\rho^p + \delta\sum_{i=1}^{n-1}\tau^{ip}+\delta 2^{-p}\tau^{np},\Tilde t +\rho^p +\delta\sum_{i=1}^{n}\tau^{ip}\Bigr).
\end{equation}
Let $n$ be such that $\tau^{n-1}\rho\leq \frac{9}{4}r\leq \tau^n \rho$. Since $d(\Tilde x, x_0)\leq \frac{5}{4}r$ we have $B(\Tilde x, \tau^{n+1}\rho)\subset B(x_0,4r)$, and hence the applications of \Cref{cor: slightly adjusted exp of pos} were justified. Moreover, we have $B(x_0,r)\subset B(\Tilde x, \frac{9}{4}r)\subset B(\Tilde x, \tau^{n}\rho)$, so that \eqref{eq: iterated expansion of pos} implies
\begin{equation}
    \label{eq: expanded pos before choice of time}
    v \geq \lambda^n\lambda_0 n(r_0) \text{ in } B(x_0,r)\times\Bigl(\Tilde t +\rho^p + \delta\sum_{i=1}^{n-1}\tau^{ip}+\delta 2^{-p}\tau^{np},\Tilde t +\rho^p +\delta\sum_{i=1}^{n}\tau^{ip}\Bigr).
\end{equation}
Let $t_m\in (t_0+2^{-p}r^p, t_0+r^p)$ be such that 
\begin{equation}
    \label{eq: choosing minimum time}
     \essinf_{B(x_0,r)} v(t_m) \leq 2\essinf_{Q^+} v. 
\end{equation}
We choose $\delta$ so that $\Tilde t +\rho^p+\rho^p \delta\sum_{i=1}^{n}\tau^{ip}=t_m$. Using elementary inequalities for geometric sums and the choice of $n$, we obtain
\begin{equation*}
    (2r)^p\leq \Bigl(\frac{9}{4}r\Bigr)^p\leq \rho^p \tau^{pn}\leq \rho^p\sum_{i=1}^{n} \tau^{ip} \leq \rho^p\tau^{p(n+1)}\leq \Bigl(\frac{9}{4}r \tau^2\Bigr)^p = \Bigl(\frac{10}{4} r\Bigr)^p.
\end{equation*}
Moreover, we have $2^{-p}r^p\leq t_m-\Tilde t - \rho^p \leq 2 r^p$. Thus we may deduce that $5^{-p}\leq \delta \leq 1$,
so that all constants are bounded independent of the choice of $\delta$. From \eqref{eq: expanded pos before choice of time} and \eqref{eq: choosing minimum time} and our choice of $\delta$ we conclude that
\begin{equation*}
    2\essinf_{Q^+}v \geq \lambda^n\lambda_0 n(r_0).
\end{equation*}
Moreover, we have
\begin{equation*}
    n(r_0)=\Bigl( \frac{8\rho}{r}\Bigr)^{-\beta} \geq 18^{-\beta}\tau^{n\beta}.
\end{equation*}
If $\beta=-\log_{\tau}\lambda$, then
\begin{equation*}
    2\essinf_{Q^+}v \geq \lambda^n \lambda_0 n(r_0) \geq 18^{-\beta} \lambda_0.
\end{equation*}
Scaling back to $u$ we have $u(z_0)\leq c\essinf_{Q^+} u$.
The claim follows by taking supremum over $z_0\in Q^-$.
\end{proof}

\section{Existence for a Cauchy problem}
\label{section: existence}
We consider the Cauchy problem associated with the doubly nonlinear equation \eqref{eq: PDE}, formally written as
\begin{equation}
\label{eq: Cauchy problem}
    \begin{cases}
       \partial_t (\lvert u\rvert^{p-2} u)-\div (\lvert \nabla u\rvert^{p-2}\nabla u )=0&\mbox{in }X\times(0,T),\\
       u(x,0)= u_0(x) &\mbox{in } X.
     \end{cases}
\end{equation}
in a metric measure space setting.
Solutions of \eqref{eq: Cauchy problem} are understood in the variational sense.
For signed functions we apply the convention \eqref{eq: signconv}.

\begin{definition}
Let $u_0\in L^p(X)$, $1<p<\infty$. A function $u\in L^p(0,T;N^{1,p}(X))\cap C([0,T];L^p(X))$ is a  solution of \eqref{eq: Cauchy problem} if $u(0)=u_0$ and 
\begin{equation}
    \label{eq: var ineq for Neumann}
   p\int_{X_T} u^{p-1}\partial_t{\varphi}\dnu  + \int_{X_T} g_u^p\dnu \leq \int_{X_T} g_{u+\varphi}^p\dnu,
\end{equation}
for every $\varphi\in L^p(0,T;N^{1,p}(X))$ with $\partial_t \varphi \in L^p(\Omega_T)$ and with compact support in time.
\end{definition}
We introduce the following compactness condition, which will be used to prove the existence of solutions to \eqref{eq: Cauchy problem}. This criterion will be used to control the limit of powers of solutions to the discretized problem, which appear in the proof of \Cref{thm: Cauchy existence} for $p\neq 2$. We have chosen this criterion since it holds on our example of Riemannian manifolds, but different conditions guaranteeing existence could be imposed as well.
\begin{definition}
    \label{def: rellich kondrachov property}
    Let $1<p<\infty$. We say that the Rellich--Kondrachov property holds in $(X,d,\mu)$, if there exists an open cover $U_i$, $i\in I$, of $X$, such that the embedding $N^{1,p}(U_i)\hookrightarrow L^p(U_i)$ is compact for each $i\in I$.
\end{definition}

Next we establish the existence of solutions to the Cauchy problem \eqref{eq: Cauchy problem} under the assumption that the Rellich--Kondrachov property holds in $(X,d,\mu)$ and that the Newton--Sobolev space $N^{1,p}(X)$ is reflexive. The reflexivity of $N^{1,p}(X)$ is applied to conclude strongly measurability of the solution $u$ as a function $u\in L^p(0,T;N^{1,p}(X))$. As a by-product we also obtain mass conservation inequalities for solutions and their upper gradients.
\begin{theorem}\label{thm: Cauchy existence}
Assume that $N^{1,p}(X)$ is reflexive, and for $p\neq 2$ assume that the Rellich--Kondrachov property from \Cref{def: rellich kondrachov property} holds. Assume that $u_0 \in N^{1,p}(X)$ is a nonnegative function. Then there exists a nonnegative solution $u$ of \eqref{eq: Cauchy problem}, which satisfies
\begin{equation}
    \label{eq: mass conserv}
    \int_{X} u(t)^{p-1}\dmu \leq \int_{X} u_0^{p-1}(t)\dmu
\end{equation}
and
\begin{equation}
    \label{eq: neumann problem energy decrease}
    \int_X g_{u(t)}^p\dmu \leq \int_{B} g_{u_0}^p\dmu
\end{equation}
for every $0<t\le T$.
\end{theorem}
\begin{proof}
Let $k\in\mathbb N$ and let $h_k=\frac{T}{k}$, $t_i= i h_k$ for every $i\in\{0,\dots, k\}$. 
Let $u_{k,0}=u_0$. 
We proceed recursively.
Suppose that, for some $i\in\{1,\dots,k\}$, we have selected a nonnegative function $u_{k,{i-1}}\in L^p(X)$.
We consider the variational integral
\begin{equation*}
    F_{k,i}(u)  = \int_X g_u^p \dmu+ \frac{1}{h_k} \int_X \bigl(\lvert u\rvert^p  -p u_{k,i-1}^{p-1} u\bigr)\dmu,
\end{equation*}
where $u\in N^{1,p}(X)$.
Since $F_{k,i}$ is coercive on $N^{1,p}(X)$, there exists a minimizer $u_{k,i}$. Since $u_{k,i-1}$ is nonnegative, we have $F_{k,i}(u_+)\leq F_{k,i}(u)$ for every $u\in N^{1,p}(X)$. 
Thus, without loss of generality, we may assume that $u_{k,i}$ is nonnegative. 

Let $v\in N^{1,p}(X)$ and $h>0$. By testing the minimizing property with $u_{k,i}+h(v-u_{k,i})=(1-h) u_{k,i}+hv$, and using the convexity of minimal upper gradients, we have
\begin{equation*}
\begin{split}
    &\int_X g_{u_{k,i}}^p \dmu + \frac{1}{h_k}\int_X\bigl(  u_{k,i}^p - p u^{p-1}_{k,i-1} u_{k,i}\bigr)\dmu
    \leq \int_X g_{(1-h)u_{k,i}+h v}^p \dmu\\
    &\qquad\qquad+ \frac{1}{h_k}\int_X\bigl( \lvert (1-h)u_{k,i}+h v\rvert^p - p u^{p-1}_{k,i-1}( (1-h)u_{k,i}+hv)\bigr)\dmu \\
    &\qquad \leq \int_X\bigl( (1-h)g_{u_{k,i}}^p+hg_v^p\bigr) \dmu\\
    &\qquad\qquad+ \frac{1}{h_k}\int_X \bigl(\lvert (1-h)u_{k,i}+h v\rvert^p - p u^{p-1}_{k,i-1}( (1-h)u_{k,i}+hv)\bigr)\dmu.
\end{split}
\end{equation*}
After rearranging the terms, we arrive at
\begin{align*}
    &\int_X g_{u_{k,i}}^p \dmu + \frac{1}{h_k} \int_X \Bigl(\frac{1}{h}( u_{k,i}^p - \lvert (1-h)u_{k,i}+h v\rvert^p) - p u^{p-1}_{k,i-1} ( u_{k,i}-v)\Bigr)\dmu \\
    &\qquad\leq \int_X g_v^p \dmu,
\end{align*}
and by letting $h\to 0$, we conclude that 
\begin{equation}
    \label{eq: tested var ineq for discretization}
    \int_X g_{u_{k,i}}^p \dmu + \frac{p}{h_k}\int_X (u_{k,i}^{p-1} - u_{k,i-1}^{p-1})(u_{k,i}-v)\dmu \leq \int_X g_v^p\dmu.
\end{equation}

Next we apply \eqref{eq: tested var ineq for discretization} to estimate $u_{k,i}$. Let $h>0$, $\ell\in\mathbb N$ and let $v=u_{k,i}-h\min\{\ell u_{k,i},1\}$. Testing \eqref{eq: tested var ineq for discretization} with $v$, we obtain
\begin{equation*}
\begin{split}
    &\int_X g_{u_{k,i}}^p \dmu + \frac{hp}{h_k}\int_X (u_{k,i}^{p-1} - u_{k,i-1}^{p-1}) \min\{\ell u_{k,i},1\}\dmu \\
    &\qquad\leq \int_{X\cap \{\ell u_{k,i}>1\}} g_u^p\dmu + \lvert 1-h\ell\rvert^p\int_{X\cap \{\ell u_{k,i}<1\}} g_u^p\dmu.
\end{split}
\end{equation*}
Rearranging terms gives
\begin{equation*}
    \frac{p}{h_k}\int_X (u_{k,i}^{p-1} - u_{k,i-1}^{p-1}) \min\{\ell u_{k,i},1\}\dmu \leq \frac{1}{h}(\lvert 1-h\ell\rvert^p-1)\int_{X\cap \{\ell u_{k,i}<1\}} g_u^p\dmu.
\end{equation*}
Letting $h\to 0$, we have
\begin{equation*}
    \lim_{h\to 0}\frac{1}{h}(\lvert 1-h\ell\rvert^p-1)\int_{X\cap \{\ell u_{k,i}<1\}} g_u^p\dmu = -p\ell^{p-1}\int_{X\cap \{\ell u_{k,i}<1\}} g_u^p\dmu \leq 0,
\end{equation*}
so that 
\begin{equation*}
    \frac{p}{h_k}\int_X (u_{k,i}^{p-1} - u_{k,i-1}^{p-1}) \min\{\ell u_{k,i},1\}\dmu\leq 0,
\end{equation*}
from which we conclude that
\begin{equation*}
    \int_X u_{k,i}^{p-1}\min\{\ell u_{k,i},1\}\dmu\leq \int_X u_{k,i-1}^{p-1}\dmu.
\end{equation*}
Letting $\ell\to\infty$, we deduce that
\begin{equation*}
    \int_X u_{k,i}^{p-1}\dmu\leq \int_X u_{k,i-1}^{p-1}\dmu.
\end{equation*}
A recursive application of the inequality above shows that
\begin{equation}
    \label{eq: discrete ineq for mass preserv}
    \int_{X} u_{k,i}^{p-1}\dmu \leq \int_{X} u_0^{p-1}\dmu.
\end{equation}

For $u,v\in \mathbb R$, we denote
\begin{equation*}
     b(u,v) = \lvert u\rvert^p-\rvert v\rvert^p-p v^{p-1}(u-v).
\end{equation*}
By \cite[Lemma 2.5]{BogeleinDuzaarMarcelliniScheven}, there exists a constant $c$, depending only on $p$, such that
\begin{equation}
    \label{eq: inequality for b}
    c\lvert u^\frac{p}{2}-v^\frac{p}{2}\rvert^2 \leq  b(u,v) \leq (u^{p-1}-v^{p-1})(u-v)
\end{equation} 
for every $u,v\geq 0$.
Applying \eqref{eq: tested var ineq for discretization} with $v=u_{k,i-1}$ and using \eqref{eq: inequality for b}, we have
\begin{equation}
\label{eq: eq for differences and energy}
    \int_X g_{u_{k,i}}^p \dmu+ \frac{1}{h_k} \int_X b(u_{k,i},u_{k,i-1})\dmu \leq \int_X g_{u_{k,i-1}}^p\dmu.
\end{equation}
A recursive application of \eqref{eq: eq for differences and energy} gives
\begin{equation}
    \label{eq: iterated est for difference quotients}
    \sum_{i=1}^k \frac{1}{h_k}\int_X  b(u_{k,i},u_{k,i-1})\dmu \leq \int_X g_{u_0}^p\dmu.
\end{equation}
Moreover, we may apply \eqref{eq: eq for differences and energy} to conclude that 
\begin{equation}
    \label{eq: discrete energy decrease}
    \int_X g_{u_{k,i}}^p\dmu \leq \int_X g_{u_0}^p\dmu,
\end{equation}
for every $i\in \{0,\dots, k\} $. Using \eqref{eq: tested var ineq for discretization} with $v=0$ and applying Young's inequality, we obtain
\begin{equation*}
\begin{split}
    \int_X g_{u_{k,i}}^p \dmu + \frac{p}{h_k} \int_X \lvert  u_{k,i} \rvert^p \dmu & \leq \frac{p}{h_k} \int_X u_{k,i-1}^{p-1} u_{k,i}\dmu \\
    & \leq \frac{1}{h_k} \int_X \bigl((p-1)\lvert u_{k,i-1}\rvert^{p} + \lvert u_{k,i}\rvert^p\bigr)\dmu,
\end{split}
\end{equation*}
from which it follows that
\begin{equation*}
    \int_X g_{u_{k,i}}^p \dmu + \frac{p-1}{h_k} \int_X \lvert  u_{k,i} \rvert^p \dmu \leq \frac{p-1}{h_k} \int_X\lvert u_{k,i-1}\rvert^{p}\dmu.
\end{equation*}
The previous estimate implies that
\begin{equation}
    \label{eq: bound for gradient and sup in exist}
    \sum_{i=1}^k h_k \int_X g_{u_k,i}^p\dmu + \max_{i\in \{1,\dots,k\}} \int_X \lvert u_{k,i}\rvert^p\dmu\leq c\int_X \lvert u_0\rvert^p\dmu.
\end{equation}

Let $u_k:X\times(-h_k,T]\to\mathbb R$, $u_{k}(\cdot,t)=u_{k,i}$, $(i-1)h_k<t\le ih_k$, $i\in \{0,\dots, k\}$.
From \eqref{eq: bound for gradient and sup in exist} we deduce that 
\begin{equation*}
    \sup_{k\in \mathbb N} \Bigl(\sup_{t\in [0,T]} \int_X \lvert u_k\rvert^p\dmu+ \int_{X_T} g_{u_k}^p\dnu\Bigr) < \infty.
\end{equation*}
Hence, the sequence $(u_k)$ is bounded in $L^\infty(0,T;L^p(X))$ and in $L^p(0,T;N^{1,p}(X))$. By reflexivity of $N^{1,p}(X)$, we conclude that there exists $u\in L^p(0,T;N^{1,p}(X))\cap L^\infty(0,T;L^p(X))$ such that, after passing to a subsequence if necessary, again denoted by $(u_k)$, we have $u_k\rightharpoonup u$ weakly in $L^p(0,T;N^{1,p}(X))$ as $k\to\infty$. 
We claim that $u_k^{p-1}\rightharpoonup u^{p-1}$ weakly in $L^{p'}(X_T)$ as $k\to\infty$,
after passing to a further subsequence if necessary. Note that for $p=2$ this is immediate. For $p\neq 2$, we conclude that there exists $w\in L^{p'}(X_T)$ such that $u_{k}^{p-1}\to w$ weakly in $L^{p'}(X_T)$. 

It remains to show that $w=u^{p-1}$. We will apply Simon's compactness result and verify that the assumptions in \cite[Theorem 1]{Simon1987} are satisfied. From \eqref{eq: iterated est for difference quotients} we have
\begin{equation}
    \label{eq: primitive continuity estimate for existence}
    \sum_{i=1}^k \frac{1}{h_k}\int_X \mathfrak \lvert u_k(t_i)^\frac{p}{2}-u_k(t_{i-1})^\frac{p}{2}\rvert^2\dmu \leq M.
\end{equation}
Let $h>0$. If $h\le h_k$, then 
\begin{equation*}
\begin{split}
    &\int_0^{T-h} \| u_k(t+h)^\frac{p}{2}-u_k(t)^\frac{p}{2}\|_{L^{2}(X)}^2 \dt \\
    &\qquad \leq \sum_{i=1}^{k} \int_{t_i-h}^{t_i} \| u_k(t_{i})^\frac{p}{2}-u_k(t_{i-1})^\frac{p}{2}\|_{L^{2}(X)}^2 
    \leq h h_k M.
\end{split}
\end{equation*}
Then we assume that $h>h_k$. Let $m\in\mathbb N$ be such that $m h_k< h < (m+1)h_k $. For $0<t<T-h$, we apply H\"older's inequality to obtain
\begin{equation*}
\begin{split}
    \| u_k(t+h)^\frac{p}{2}-&u_k(t)^\frac{p}{2} \|_{L^2(X)}
    \leq \sum_{j=1}^{m} \| u_k(t+jh_k)^\frac{p}{2}-u_k(t+(j-1)h_k)^\frac{p}{2} \|_{L^2(X)} \\
    &\qquad+\| u_k(t+h)^\frac{p}{2}- u_k(t+m h_k)^\frac{p}{2}\|_{L^2(X)} \\
    &\leq  m^\frac{1}{2}\Bigl(\sum_{j=1}^m \| u_k(t+jh_k)^\frac{p}{2}-u_k(t+(j-1)h_k)^\frac{p}{2} \|_{L^2(X)}^2\Bigr)^\frac{1}{2}\\
    &\qquad+ \| u_k(t+h)^\frac{p}{2}- u_k(t+m h_k)^\frac{p}{2}\|_{L^2(X)}.
\end{split}  
\end{equation*}
We integrate the estimate above over $0<t<T-h$ and apply $m<\frac{h}{h_k}$ to have
\begin{equation*}
\begin{split}
    &\int_0^{T-h} \| u_k(t+h)^\frac{p}{2}-u_k(t)^\frac{p}{2}\|_{L^{2}(X)}^2 \dt\\ &\qquad\leq \frac{c h}{h_k} \int_{0}^{T-h} \sum_{j=1}^m\| u_k(t+j h_k)^\frac{p}{2}-u_k(t+(j-1)h_k)^\frac{p}{2}\|_{L^{2}(X)}^2\dt \\ 
    & \qquad\qquad + c\int_0^{T-h} \| u_k(t+h)^\frac{p}{2}- u_k(t+m h_k)^\frac{p}{2}\|_{L^2(X)}\dt \\
    & \qquad\leq cTM h + c T M h_k\leq c h. 
\end{split}
\end{equation*}
Hence, we conclude that 
\begin{equation}
    \label{eq: conv of shifts}
    \|u_{k}(t+h)^\frac{p}{2}-u_k(t)^\frac{p}{2}\|_{L^2(0,T-h;L^2(X))}\to 0
\end{equation}
as $h\to0$ uniformly in $k$. 

Let $U_i$, $i\in I$ be a set as in \Cref{def: rellich kondrachov property}. By the compactness of the embedding $N^{1,p}(U_i)\hookrightarrow L^p(U_i)$, \eqref{eq: discrete energy decrease} and \eqref{eq: bound for gradient and sup in exist} we have that $\int_{t_1}^{t_2}u_k(t)^\frac{p}{2}\dt$ belongs to a compact subset of $L^2(U_i)$ for every $0<t_1,t_2<T$ and $k\in\mathbb N$. Thus, we use \eqref{eq: conv of shifts} and apply \cite[Theorem 1]{Simon1987} to conclude that, after passing to a subsequence if necessary, $u_k\to u$ in $L^p((U_i)_T)$ as $k\to\infty$. This shows that $w=u^{p-1}$ in $U_i$. Since $i\in I$ was arbitrary and $X\subset \bigcup_{i\in I} U_i$, the claim follows.

It follows that $u_k^{p-1}\rightharpoonup u^{p-1}$ weakly in $L^1(B)$ as $k\to\infty$ for any ball $B=B(x_0,r)$ in $X$. From \eqref{eq: discrete ineq for mass preserv} we have
\begin{equation*}
    \int_X u_k(t)^{p-1}\dmu\leq \int_X u_0^{p-1}\dmu.
\end{equation*}
Let $\delta>0$. Integrating the above over $\tau<t<\tau+\delta$, we have
\begin{equation*}
    \frac{1}{\delta}\int_{X\times(\tau,\tau+\delta)} u_k^{p-1}\dnu \leq \int_X u_0^{p-1}\dmu.
\end{equation*}
Thus for any ball $B=B(x_0,r)$ we have
\begin{equation*}
    \frac{1}{\delta}\int_{B\times(\tau,\tau+\delta)} u^{p-1}\dnu\leq \liminf_{k\to\infty}\frac{1}{\delta}\int_{X\times(\tau,\tau+\delta)} u_k^{p-1}\dnu \leq \int_X u_0^{p-1}\dmu.
\end{equation*}
Since the ball $B$ was arbitrary we conclude that
\begin{equation*}
    \frac{1}{\delta}\int_{X\times(\tau,\tau+\delta)} u^{p-1}\dnu\leq \int_X u_0^{p-1}\dmu.
\end{equation*}
Letting $\delta\to 0$, we obtain
\begin{equation*}
    \int_{X} u(\tau)^{p-1}\leq \int_X u_0^{p-1}\dmu,
\end{equation*}
which implies \eqref{eq: mass conserv}.

Arguing similarly as before, we conclude that $u_k^\frac{p}{2}\rightharpoonup u^\frac{p}{2}$ weakly in $L^2(X_T)$ as $k\to\infty$. Again, this is immediate if $p=2$. We have shown that 
\begin{equation*}
    w_k(t)=\frac{1}{h_k}\bigl(u_k(t+h_k)^\frac{p}{2}-u_k(t)^\frac{p}{2}\bigr),
\end{equation*}
is bounded in $L^2(X\times(0,T-h_k))$, independently of $k\in\mathbb N$. Hence, after passing to a further subsequence, the sequence $(w_k)$ has a weak limit $w$ in $L^2(X_T)$. For every $\varphi\in C_0^\infty([0,T];L^p(X))$ we have
\begin{equation*}
\begin{split}
    \int_{X_T} u^\frac{p}{2} \partial_t \varphi \dnu & = -\lim_{k\to\infty} \int_{X_T} u_k(t)^\frac{p}{2}\frac{\varphi(t)-\varphi(t-h_k)}{h_k}\dnu \\
    & = -\lim_{k\to\infty} \int_{X_T} \frac{u_k(t+h_k)^\frac{p}{2}-u_k(t)}{h_k} \varphi(t)\dnu \\
    & = -\int_{X_T} w \varphi\dnu,
\end{split}
\end{equation*}
which shows that $\partial_t u^\frac{p}{2}=w\in L^2(X_T)$. This implies $u^\frac{p}{2}\in C([0,T];L^2(X))$ and hence $u\in C([0,T];L^p(X))$.

Let $v\in L^p(0,T;N^{1,p}(X))$ with $\partial_t v\in L^p(X_T)$, and let $0<\tau<T$. Integrating \eqref{eq: tested var ineq for discretization} over $0<t<\tau$ gives
\begin{equation*}
    \int_{X_\tau} g_{u_{k}}^p \dnu + \frac{p}{h_k}\int_{X_\tau} (u_{k}^{p-1} - u_{k}(\cdot-h_k)^{p-1})(u_{k}-v)\dnu 
    \leq \int_{X_\tau} g_v^p\dnu.
\end{equation*}
By Young's inequality 
\begin{equation*}
\begin{split}
    &\frac{p}{h_k}\int_{X_\tau}\bigl(u_k^{p-1}-u_k(\cdot-h_k)^{p-1}(-u_k)\bigr)\dnu\\
    &\qquad= \frac{p}{h_k}\int_{X_T}\bigl( u_{k}(\cdot-h_k)^{p-1} u_k - \lvert u_k\rvert^p\bigr)\dnu \\
    &\qquad \leq \frac{p}{h_k}\int_{X_\tau}\Bigl(\frac{p-1}{p}\lvert u_k\rvert^p(\cdot-h_k) + \frac{1}{p} \lvert u_k\rvert^p - \lvert u_k\rvert^p\Bigr)\dnu \\ 
    &\qquad = \frac{p-1}{h_k} \int_{X_\tau} \bigl(\lvert u_k\rvert^p(\cdot-h_k)-\lvert u_k\rvert^p\bigr)\dnu \\
    &\qquad = (p-1)\int_{X} \lvert u_0\rvert^p\dmu - \frac{p-1}{h_k}\int_{X\times(\tau-h_k,\tau)}\lvert u_k\rvert^p\dnu.
\end{split}
\end{equation*}
We also have
\begin{equation*}
\begin{split}
    &\frac{p}{h_k}\int_{X_\tau}(u_k^{p-1}-u_k(\cdot-h_k)^{p-1})v \dnu\\
    &\qquad= \frac{p}{h_k}\int_{X_\tau} u_k^{p-1} v \dnu -\frac{p}{h_k}\int_{X_\tau} u_k(\cdot-h_k)^{p-1} v\dnu \\
    &\qquad = \frac{p}{h_k} \int_{X_\tau} u_k^{p-1} v \dnu - \frac{p}{h_k}\int_{X\times(-h_k,\tau-h_k)} u_k^{p-1} v(\cdot+h_k)\dnu\\
    &\qquad = - \frac{p}{h_k}\int_{X_\tau}  (v(\cdot+h_k)-v) u_k^{p-1}\dnu-\frac{p}{h_k}\int_{X\times(-h_k,0)} u_0^{p-1}v(\cdot+h_k)\dnu \\
    &\qquad \qquad + \frac{p}{h_k}\int_{X\times(\tau-h_k,\tau)}u_k^{p-1}v(\cdot+h_k)\dnu.
\end{split}
\end{equation*}
The previous estimates imply 
\begin{equation}
\label{eq: est for limiting var ineq and initial val}
\begin{split}
    & \frac{p-1}{h_k}\int_{X\times(\tau-h_k,\tau)}\lvert u_k\rvert^p\dnu+\int_{X_\tau} g_{u_{k}}^p \dnu
    \leq \int_{X_\tau} g_v^p\dnu\\
    &\qquad-\frac{p}{h_k}\int_{X_\tau}  (v(\cdot+h_k)-v) u_k^{p-1}\dnu
    -\frac{p}{h_k}\int_{X\times(-h_k,0)} u_0^{p-1}v(\cdot+h_k)\dnu\\
    &\qquad+ \frac{p}{h_k}\int_{X\times(\tau-h_k,\tau)}u_k^{p-1}v(\cdot+h_k)\dnu.
\end{split}
\end{equation}

Let us first use \eqref{eq: est for limiting var ineq and initial val} to show that $u$ has the correct initial values. We test \eqref{eq: est for limiting var ineq and initial val} with $v=u_0$ and obtain
\begin{equation*}
\begin{split}
    & \frac{p-1}{h_k}\int_{X\times(\tau-h_k,\tau)}\lvert u_k\rvert^p\dnu+\int_{X_\tau} g_{u_{k}}^p \dnu
    \leq \tau \int_{X} g_{u_0}^p\dmu\\
    &\qquad-\int_{X} u_0^{p}\dmu + \frac{p}{h_k}\int_{X\times(\tau-h_k,\tau)}u_k^{p-1}u_0\dnu.
\end{split}
\end{equation*}
Rearranging and using \eqref{eq: inequality for b}, we conclude that
\begin{equation*}
\begin{split}
    \frac{1}{h_k}\int_{X\times(\tau-h_k,\tau)} b(u_0,u_k)\dnu 
    &\leq  \frac{1}{h_k}\int_{X\times(\tau-h_k,\tau)}\bigl((p-1)\lvert u_k\rvert^p-pu_k^{p-1}u_0 + u_0^p\bigr)\dnu \\
    &\leq \tau\int_{X} g_{u_0}^p\dmu  .
\end{split}
\end{equation*}
Let $\delta>0$. Integrating the estimate above over $\tau<t<\tau+\delta$, and using  $ b(u_0,u_k)\geq 0$, we obtain
\begin{equation*}
    \frac{1}{\delta}\frac{1}{h_k}\int_\tau^{\tau+\delta}\int_{X\times(t-h_k,t)} b(u_0,u_k)\dmu\ds\dt \leq \tau\int_{X} g_{u_0}^p\dmu.
\end{equation*}
For $h_k<\delta$, we have
\begin{equation*}
    \frac{1}{\delta}\int_\tau^{\tau+\delta-h_k} b(u_0,u_k)\dnu \leq \tau\int_{X} g_{u_0}^p\dmu.
\end{equation*}
Letting $k\to\infty$ and then $\delta \to 0$, we conclude that
\begin{equation*}
    \int_{X} b(u_0,u(\tau))\dmu \leq \tau\int_{X} g_{u_0}^p\dmu
\end{equation*}
for every $0<\tau<T$. This implies that $ b(u_0,u(t))\to 0 $ as $t\to 0$, which in turn implies that $u(t)\to u_0$ in $L^p(X)$ as $t\to 0$.

Next we apply \eqref{eq: est for limiting var ineq and initial val} to show that \eqref{eq: var ineq for Neumann} holds. Let $\varphi$ be as in \eqref{eq: var ineq for Neumann}. Let $v\in L^p(0,T;N^{1,p}(X))$ with $\partial_t v\in L^p(X_T)$ such that $v$ is compactly supported in $(0,\tau)$. Testing \eqref{eq: est for limiting var ineq and initial val} with $v$ we conclude that
\begin{equation}
\label{eq: var ineq before averaging and limit}
\begin{split}
    & \frac{p-1}{h_k}\int_{X\times(\tau-h_k,\tau)}\lvert u_k\rvert^p\dnu + \int_{X_\tau} g_{u_{k}}^p \dnu
    \leq \int_{X_\tau} g_v^p\dnu\\
    &\qquad+ (p-1)\int_{X}\lvert u_0\rvert^p\dmu - \frac{p}{h_k}\int_{X_\tau}  (v(\cdot+h_k)-v) u_k^{p-1}\dnu,
\end{split}
\end{equation}
whenever $h_k$ is small enough. Let $\delta >0$, and integrate the above over $t_0<\tau<t_0+\delta$ to deduce that
\begin{equation*}
\begin{split}
    & \frac{p-1}{\delta}\int_{X\times(t_0,t_0+\delta-h_k)}\lvert u_k\rvert^p\dnu + \int_{X_\tau} g_{u_{k}}^p \dnu
    \leq \int_{X_{\tau+\delta}} g_v^p\dnu\\
    &\qquad+ (p-1)\int_{X}\lvert u_0\rvert^p\dmu - \int_{t_0}^{t_0+\delta}\frac{p}{h_k}\int_{X_\tau}  (v(\cdot+h_k)-v) u_k^{p-1}\dmu\dt\,d\tau,
\end{split}
\end{equation*}
for $h_k<\delta$. Letting $k\to \infty$ gives
\begin{equation*}
\begin{split}
    & \frac{p-1}{\delta}\int_{X\times(t_0,t_0+\delta)} u^p\dnu + \int_{X_\tau} g_{u}^p \dnu 
    \leq \int_{X_{\tau+\delta}} g_v^p\dnu\\
    &\qquad+ (p-1)\int_{X}\lvert u_0\rvert^p\dmu - \int_{t_0}^{t_0+\delta}\frac{p}{h_k}\int_{X_\tau}  \partial_t v u^{p-1}\dmu\dt\,d\tau,
\end{split}
\end{equation*}
where we used the weak convergence of $u_k^{p-1}\rightharpoonup u^{p-1}$ in $L^{p'}(X_T)$ as $k\to\infty$. Letting $\delta \to 0$ and taking $t_0=\tau$, we conclude that
\begin{equation}
\label{eq: exist var ineq before time moll}
\begin{split}
    & (p-1)\int_{X}u(\tau)^p\dnu + \int_{X_\tau} g_{u}^p \dnu \\
    & \qquad\leq \int_{X_\tau} g_v^p\dnu  + (p-1)\int_{X}\lvert u_0\rvert^p\dmu - p \int_{X_\tau} \partial_t v u^{p-1}\dnu.
\end{split}
\end{equation}
Let $\delta>0$, $\epsilon>0$, and consider a cutoff function
\begin{equation*}
    \eta(t)
        =\begin{cases}
            0 ,&\quad t\leq \epsilon,\\
            \frac{1}{\epsilon}(t-2\epsilon), & \quad \epsilon <t\leq 2\epsilon, \\
            1,&\quad 2\epsilon < t \leq \tau-2\epsilon,\\
            \frac{1}{\epsilon}(\tau-2\epsilon-t), & \quad \tau-2\epsilon<t\leq \tau-\epsilon,\\
            0,&\quad \tau-\epsilon<t.
        \end{cases}
\end{equation*}
Testing \eqref{eq: exist var ineq before time moll} with $v=\eta (u_\delta +\varphi)$, we infer that
\begin{equation}
\label{eq: converting between forms of weak sol}
\begin{split}
    & (p-1)\int_{X}u(\tau)^p\dmu + \int_{X_\tau} g_{u}^p \dnu
    \leq \int_{X_\tau} g_{ \eta u_\delta + \varphi}^p\dnu\\
    &\qquad + (p-1)\int_{X}\lvert u_0\rvert^p\dmu - p \int_{X_\tau} \partial_t (\eta u_\delta +\varphi) u^{p-1}\dnu.
\end{split}
\end{equation}
For the term containing the time derivative, we have
\begin{equation*}
\begin{split}
    &- p \int_{X_\tau} \partial_t (\eta u_\delta +\varphi) u^{p-1}\dnu  = -p \int_{X_\tau} \partial_t \varphi u^{p-1}\dnu-p \int_{X_\tau} \partial_t \eta u_\delta  u^{p-1}\dnu \\
    & \qquad\qquad-p \int_{X_\tau} \eta \partial_t u_\delta  (u^{p-1}-(u_\delta)^{p-1})\dnu -p \int_{X_\tau} \eta \partial_t u_\delta (u_{\delta})^{p-1}\dnu \\
    &\qquad  \leq -p \int_{X_\tau} \partial_t \varphi u^{p-1}\dnu-p \int_{X_\tau} \partial_t \eta u_\delta  u^{p-1}\dnu 
     + \int_{X_\tau} \partial_t\eta (u_\delta)^p \dnu,
\end{split}
\end{equation*}
where in the second inequality we used \eqref{eq: time der of mollification} and an integration by parts. We use the estimate above and \Cref{lemma: convergence of time mollification} to pass $\delta\to 0$ in \eqref{eq: converting between forms of weak sol} and conclude that
\begin{equation}
\label{eq: exist var ineq before eps to 0}
\begin{split}
    & (p-1)\int_{X}u(\tau)^p\dmu + p\int_{X_T} \partial_t\varphi u^{p-1}\dnu + \int_{X_\tau} g_{u}^p \dnu \\
    & \qquad\leq \int_{X_T} g_{ \eta u + \varphi}^p\dnu  + (p-1)\int_{X}\lvert u_0\rvert^p\dmu - (p-1) \int_{X_T} \partial_t \eta u^{p}\dnu.
\end{split}
\end{equation}
By the definition of $\eta$, we have
\begin{equation*}
    (p-1) \int_{X_T} \partial_t \eta u^{p}\dnu= \frac{p-1}{\epsilon} \int_{X\times(\epsilon,2\epsilon)} u^{p}\dnu - \frac{p-1}{\epsilon} \int_{X\times(\tau-\epsilon,\tau-2\epsilon)} u^{p}\dnu.
\end{equation*}
Hence, letting $\epsilon \to 0$, from \eqref{eq: exist var ineq before eps to 0} we conclude that
\begin{equation*}
\begin{split}
    p\int_{X_T} \partial_t\varphi u^{p-1}\dnu + \int_{X_\tau} g_{u}^p \dnu \leq \int_{X_T} g_{ \eta u + \varphi}^p\dnu,
\end{split}
\end{equation*}
which shows that \eqref{eq: var ineq for Neumann} holds. Finally, \eqref{eq: discrete energy decrease} shows that 
\begin{equation*}
    \int_X g_{u_k}^p(t)\dmu\leq \int_{X} g_{u_0}^p\dmu,
\end{equation*}
from which we conclude that \eqref{eq: neumann problem energy decrease} holds.
\end{proof}

\section{Parabolic Harnack implies doubling}
\label{section: PHI implies VD}
Theorem \ref{thm: PHI} asserts that a scale and location invariant parabolic Harnack inequality holds for nonnegative solutions of \eqref{eq: PDE}, under the assumptions that the measure $\mu$ is doubling and that a  $(1,p)$-Poincar\'e inequality holds in $(X,d,\mu)$.
In this section we discuss the converse direction and show that a parabolic Harnack inequality for nonnegative solutions of \eqref{eq: PDE} implies that the measure $\mu$ is doubling. 

\begin{definition}\label{def: PHI}
The parabolic Harnack inequality holds in $(X,d,\mu)$ if there exists a constant $c_H>0$ such that whenever $u$ is a nonnegative solution of \eqref{eq: PDE} in $Q=B(x_0,4r)\times (t_0-(4r)^p,t_0+(4r)^p)$, then
\begin{equation*}
    \esssup_{Q^-} u \leq c_H \essinf_{Q^+} u,
\end{equation*}
where $Q^-=B(x_0,r)\times (t_0- r^p,t_0-2^{-p}r^p)$
and $Q^+=B(x_0,r)\times (t_0+2^{-p}r^p,t_0+r^p)$.
\end{definition}

We apply Theorem \ref{thm: Cauchy existence} that ensures the existence of a solution to the Cauchy problem \eqref{eq: Cauchy problem} in the proof of the result below.

\begin{theorem}
\label{thm: PHI implies doubling}
Assume that the parabolic Harnack inequality holds in $(X,d,\mu)$, that $N^{1,p}(X)$ is reflexive, and for $p\neq 2$ assume that the Rellich--Kondrachov property in \Cref{def: rellich kondrachov property} holds. Then the measure $\mu$ is doubling, with a constant $c_\mu$ depending only on $c_H$ and $p$.
\end{theorem}
\begin{proof}
Let $x_0\in X$, $r>0$ and $T=(64r)^p$. Let $f\in N^{1,p}(X)$ be a nonnegative function, with $f=0$ in $X\backslash B(x_0,\frac r8)$ and 
\begin{equation*}
    \int_{X} f^{p-1}\dmu=1.
\end{equation*} 
Let $u$ be the solution of the Cauchy problem \eqref{eq: Cauchy problem}, given by \Cref{thm: Cauchy existence} with $u_0=f$. We extend $u$ as zero for $t<0$, and note that $u$ is a  solution in $(X\backslash \overline{B(x_0,\frac r8)})\times (-\infty,T)$.
We claim that $u$ a  supersolution in $X\times (-\infty,T)$. To see this, let $\varphi\in L^p(0,T;N^{1,p}(X))$ with $\partial_t \varphi\in L^p(X\times(-\infty,T))$ and $\supp\varphi\subset X\times(-\infty,T)$. Let $h>0$ and consider a cutoff function
\begin{equation*}
    \eta_h(t)
        =\begin{cases}
            0, &\quad t\leq 0, \\
            \frac{t}{h} ,&\quad 0 <t < h,\\
            1,&\quad t\geq h.
        \end{cases}
\end{equation*}
Testing \eqref{eq: var ineq} with $\eta_h\varphi$, we obtain
\begin{equation*}
\begin{split}
    &\frac{1}{h}\int_0^h p\int_{X} u^{p-1}\varphi\dnu+p\int_{X\times(0,T)} u^{p-1} \partial_t{\varphi}\eta_h\dnu  + \int_{X\times(0,T)} g_u^p\dnu \\
    &\qquad\leq \int_{X\times(0,T)} g_{u+\varphi}^p\dnu.
\end{split}
\end{equation*}
Letting $h\to 0$ implies
\begin{equation*}
\begin{split}
    &p\int_{X} f^{p-1}\varphi\dmu+p\int_{X\times(-\infty,T)} u^{p-1} \partial_t{\varphi}\dnu + \int_{X\times(-\infty,T)} g_u^p\dnu \\
    &\qquad\leq \int_{X\times(-\infty,T)} g_{u+\varphi}^p\dnu,
\end{split}
\end{equation*}
from which the claim follows.

Denote $Q_r = B(x_0,r)\times (-r^p,r^p)$.
Let $\eta$ be a $\frac{c}{r}$-Lipschitz cutoff function, with $0\leq \eta\leq 1$, $\eta=1$ on $B(x_0,\frac r4)$ and with $\supp\eta\subset B(x_0,\frac r2)$. 
Consider a cutoff function
\begin{equation*}
    \xi(t)
        =\begin{cases}
            \frac{t+(\frac r2)^p}{(1-2^{-p})},&\quad -(\frac r2)^p <t < -(\frac r4)^p,\\
            1,&\quad -(\frac r4)^p\leq t\leq (\frac r4)^p,\\
            \frac{(\frac r2)^p-t}{(1-2^{-p})(\frac r2)^p},&\quad (\frac r4)^p<t\leq (\frac r2)^p,\\
            0,&\quad\text{otherwise}.
        \end{cases}
\end{equation*}
Let $h>0$ and consider a cutoff function
\begin{equation*}
    \zeta_h(t)
        =\begin{cases}
            0 ,&\quad t \leq 0,\\
            \frac{1}{h}(t-h) ,&\quad 0<t < h,\\
            1,&\quad  t\geq h.
        \end{cases}
\end{equation*}
Let $\varphi=\eta\xi$. 
Applying $\varphi \zeta_h$ as test function in \Cref{lemma: var sol implies weak sol with inequality} gives
\begin{equation*}
    \frac{1}{h}\int_{0}^{h}\int_{B(x_0,\frac r2)}u^{p-1}\varphi\dmu \leq  \int_{Q_{\frac r2}} g_u^{p-1}g_\varphi \dnu-\int_{Q_{\frac r2}} u^{p-1}\partial_t \varphi  \dmu.
\end{equation*}
Letting $h\to 0$, we obtain
\begin{equation*}
    \int_{B(x_0,\frac r2)}f^{p-1}\eta\dmu \leq  \int_{Q_{\frac r2}} g_u^{p-1}g_\varphi \dnu-\int_{Q_{\frac r2}} u^{p-1}\partial_t \varphi \dmu,
\end{equation*}
from which it follows that
\begin{equation}
    \label{eq: est from below for doubling}
    1 =  \int_{B(x_0,\frac r2)} f^{p-1}\eta \dmu \leq  \int_{Q_{\frac r2}} g_u^{p-1}g_\varphi \dnu-\int_{Q_{\frac r2}} u^{p-1}\partial_t \varphi \dmu.
\end{equation}
To estimate the first term on the right-hand side of \eqref{eq: est from below for doubling}, we use H\"older's inequality and the fact that $\varphi=1$ on $Q_{\frac r4}$ to obtain
\begin{equation*}
\begin{split}
    \int_{Q_{\frac r2}} g_u^{p-1}g_\varphi \dnu & \leq \frac{c}{r}\int_{Q_{\frac r2}\backslash Q_{\frac r4}} g_u^{p-1} \dnu \\
    & \leq  c \mu(B(x_0,\tfrac r2))^\frac{1}{p} \biggl(\int_{Q_{\frac r2}\backslash Q_{\frac r4}}  g_u^p\dnu\biggr)^{\frac{p-1}{p}}.
\end{split}
\end{equation*}
Denote $M=\esssup_{Q_{\frac r2}\backslash Q_{\frac r4}}u$. Applying \Cref{lemma: caccioppoli with cutoff} and the fact that $g_u=g_{(M-u)_+}$ on $Q_{\frac r2}\backslash Q_{\frac r4}$, we have
\begin{equation*}
\begin{split}
    \int_{Q_{\frac r2}\backslash Q_{\frac r4}} g_u^p\dnu 
    & = \int_{Q_{\frac r2}\backslash Q_{\frac r4}} g_{(M-u)_+}^p\dnu 
     \leq \int_{Q_{\frac r2}} g_{(M-u)_+}^p\dnu \\
    & \leq \frac{c}{r^p} \int_{Q_{r}\backslash Q_{\frac r2}}\bigl((M-u)_+^p+( u+M)^{p-2}(M-u)_+^2\bigr) \dnu \\
    & \leq c \mu(B(x_0,r)) \esssup_{Q_{r}\backslash Q_{\frac r4}} u^p .
\end{split}
\end{equation*}
It follows that 
\begin{equation*}
    \int_{Q_{r/2}} g_u^{p-1}g_\varphi \dnu \leq c \mu(B(x_0,r)) \esssup_{Q_{r}\backslash Q_{r/4}} u^{p-1},
\end{equation*}
and
\begin{equation*}
    -\int_{Q_r} u^{p-1}\partial_t \varphi \dnu \leq c \mu(B(x_0,r)) \esssup_{Q_r\backslash Q_{\frac r4}} u^{p-1}.
\end{equation*}
Hence, from \eqref{eq: est from below for doubling} we conclude that
\begin{equation}
    \label{eq: lower diagonal estimate}
    \left(\frac{1}{\mu(B(x_0,r))}\right)^\frac{1}{p-1} \leq c \esssup_{Q_r\backslash Q_{\frac r4}} u.
\end{equation}
We claim that
\begin{equation}
    \label{eq: harnack for going forward}
    \esssup_{Q_{r}\backslash Q_{\frac r4}} u\leq c \essinf_{Q +} u,
\end{equation}
where 
\[
Q^+= B(x_0,2r)\times ((15r)^p+(1+2^{-p})(4r)^p, (16r)^p+2(4r)^p).
\]
To see this, let $(x,t)\in Q_{r}\backslash Q_{\frac r4}$ with 
\[
\esssup_{Q_{r}\backslash Q_{\frac r4}} u \leq 2 u(x,t).
\]
Since $u(x,s)=0$ for $s<0$, we conclude that $t>0$. Hence, $u$ is a  solution in $B(x,\frac r8)\times(t-(\frac r4)^p,T)$.

For $k\in\mathbb N$ let
\begin{equation*}
    Q_k= B(x,\tfrac{r}{32})\times (t+k(1+2^{-p})(\tfrac{r}{32})^p, t+k(1+2^{-p})(\tfrac{r}{32})^p+(\tfrac{r}{32})^p).
    \quad 
\end{equation*}
We apply the parabolic Harnack inequality to conclude that
\begin{equation*}
    \esssup_{Q_k}u\leq c_H \essinf_{Q_{k+1}}u,
\end{equation*}
for every $k\in\mathbb N$ and hence 
\begin{equation*}
    \esssup_{Q_{r}\backslash Q_{\frac r4}}u \leq 2u(x,t)\leq 2\esssup_{Q_0} u \leq 2 c_H^k \essinf_{Q_k} u.
\end{equation*}
Let $k$ be the smallest integer such that $k(1+2^{-p})(\frac{1}{32})^p>16^p$. Denote 
\begin{equation*}
    Q^-=B(x_0,4r) \times ((16r)^p, (16r)^p+(1-2^{-p})(4r)^p).
\end{equation*}
Since $0<t<r^p$ the choice of $k$ implies $Q_k \subset Q^-$.
Applying the parabolic Harnack inequality once more, we obtain
\begin{equation*}
    \esssup_{Q_{r}\backslash Q_{\frac r4}}\leq  c \essinf_{Q_k} u \leq c \esssup_{Q^-} u\leq c\essinf_{Q^+} u,
\end{equation*}
which shows that \eqref{eq: harnack for going forward} holds. Then, for $0<t<T$, we estimate
\begin{equation*}
\begin{split}
    \essinf_{B(x_0,2r)} u(t) & \leq c \biggl(\dashint_{B(x_0,2r)} u(t)^{p-1}\dmu\biggr)^\frac{1}{p-1} \\
    & \leq c \biggl(\frac{1}{\mu(B(x_0,2r))} \int_{X} u(t)^{p-1}\dmu\biggr)^\frac{1}{p-1} \\
    & \leq c \biggl(\frac{1}{\mu(B(x_0,2r))}\int_{X} f^{p-1}\dmu\biggr)^\frac{1}{p-1} \\
    & = c \biggl(\frac{1}{\mu(B(x_0,2r))}\biggr)^\frac{1}{p-1},
\end{split}
\end{equation*}
where we used \eqref{eq: mass conserv}. It follows that 
\begin{equation}
\label{eq: upper diagonal estimate}
    \essinf_{Q^+} u \leq c \biggl(\frac{1}{\mu(B(x_0,2r))}\biggr)^\frac{1}{p-1}.
\end{equation}
Combining \eqref{eq: lower diagonal estimate}, \eqref{eq: harnack for going forward} and \eqref{eq: upper diagonal estimate}, we find that
\begin{equation*}
    \biggl(\frac{1}{\mu(B(x_0,r))}\biggr)^\frac{1}{p-1} \leq c \esssup_{Q_r\backslash Q_{\frac r2}} u \leq c \essinf_{Q^+} u \leq c \biggl(\frac{1}{\mu(B(x_0,2r))}\biggr)^\frac{1}{p-1},
\end{equation*}
which proves the claim.
\end{proof}

\section{Parabolic Harnack implies Poincar\'e inequality}
\label{section: PHI implies Poincare}
In this section we show that the parabolic Harnack inequality in Definition \ref{def: PHI} implies that a $(1,p)$-Poincar\'e inequality holds in $(X,d,\mu)$.  We will make use of the fact that the parabolic Harnack inequality implies that $\mu$ is doubling, which was shown in \Cref{section: PHI implies VD}. We begin with the following Poincar\'e-type inequality for functions with zero boundary values.
\begin{proposition}
\label{prop: PHI implies local sobolev}
Assume that the parabolic Harnack inequality holds in $(X,d,\mu)$, that $N^{1,p}(X)$ is reflexive, and for $p\neq 2$ assume that the Rellich--Kondrachov property from \Cref{def: rellich kondrachov property} holds. Then there exist constants $c$ and $\gamma>0$, depending only on $c_H$ and $p$, such that 
\begin{equation}
    \label{eq: sobo type ineq from phi}
    \int_{B(x_0,r)} \lvert f\rvert^p\dmu  \leq c r^p \int_{B(x_0,r)} g_f^p\dmu,
\end{equation}
for every $f\in N^{1,p}_0(B(x_0,r))$, where $x_0\in X$ and $0<r\leq\gamma\diam X$. 
\end{proposition}
\begin{proof}
Let $f\in N^{1,p}_0(B(x_0,r))$. Without loss of generality we may assume that $f$ is nonnegative. Let $u$ be the nonnegative solution of the Cauchy problem \eqref{eq: Cauchy problem}, given by \Cref{thm: Cauchy existence} with $u_0=f$ and $T>0$ is to be chosen later. Let $\tau>0$. For $h>0$, consider a cutoff function
\begin{equation*}
    \zeta_h(t)
        =\begin{cases}
            \frac{1}{h}(t-h) ,&\quad 0<t < h,\\
            1,&\quad h\leq t\leq \tau,\\
            \frac{1}{h}(\tau+h-t),&\quad \tau<t\leq \tau+h,\\
            0,&\quad\text{otherwise}.
        \end{cases}
\end{equation*}
Using \Cref{lemma: var sol implies weak sol with inequality} with $\varphi= \zeta_h f$ gives
\begin{equation*}
\begin{split}
    & \frac{1}{h}\int_{0}^{h}\int_{B(x_0,r)} u^{p-1}f\dmu\dt 
     \leq  \int_{B(x_0,r)\times(0,\tau)} g_u^{p-1} g_f\dnu\\
     &\qquad+ \frac{1}{h}\int_{\tau}^{\tau+h}\int_{B(x_0,r)}  u^{p-1}f\dmu\dt.
\end{split}
\end{equation*}
Letting $h\to 0$, we obtain
\begin{equation*}
    \int_{B(x_0,r)} f^p\dmu \leq  \int_{B(x_0,r)\times(0,\tau)} g_u^{p-1}g_f\dnu + \int_{B(x_0,r)}  u(\tau)^{p-1}f\dmu.
\end{equation*}
By Young's inequality 
\begin{equation*}
    \int_{B(x_0,r)} f^p\dmu \leq c\int_{B(x_0,r)\times(0,\tau)} g_u^{p-1}g_f\dnu + \int_{B(x_0,r)}  u(\tau)^p\dmu.
\end{equation*}
H\"older's inequality and \eqref{eq: neumann problem energy decrease} imply that
\begin{equation*}
    \dashint_{B(x_0,r)} f^p\dmu \leq c\tau\dashint_{B(x_0,r)} g_f^{p}\dmu + \dashint_{B(x_0,r)}  u(\tau)^p\dmu.
\end{equation*}
Let $\sigma >1$. Writing the previous estimate for $\tau=(4\sigma r)^p$, we have
\begin{equation}
    \label{eq: basic estimate for sobolev}
    \dashint_{B(x_0,r)} f^p\dmu \leq c\sigma^p r^p\dashint_{B(x_0,r)} g_f^{p}\dmu + \dashint_{B(x_0,r)}  u((4\sigma r)^p)^p\dmu.
\end{equation}
Let us estimate the second term on the right-hand side of the above. Let 
\begin{equation*}
    Q^-= B(x_0,\sigma r)\times ((4\sigma r)^p, (4\sigma r)^p+(1-2^{-p})(\sigma r)^p),
\end{equation*}
and
\begin{equation*}
    Q^+= B(x_0,\sigma r)\times ((4\sigma r)^p+(1+2^{-p})(\sigma r)^p, (4\sigma r)^p+2(\sigma r)^p)
\end{equation*}
The parabolic Harnack inequality implies
\begin{equation}
\label{eq: application of PHI for Sobolev}
\begin{split}
    \dashint_{B(x_0,r)}  u((4\sigma r)^p)^p\dmu  
    &\leq \esssup_{Q^-} u^p 
    \leq c \essinf_{Q^+} u^p \\
    &\leq c \biggl(\dashint_{Q^+} u^{p-1}\dnu\biggr)^{\frac{p}{p-1}}.
\end{split}
\end{equation}
For any $t>0$ we have 
\begin{equation*}
\begin{split}
    \int_{B(x_0,\sigma r)} u(t)^{p-1}\dmu
    &\leq \int_{X} u(t)^{p-1}\dmu \\
    &\leq \int_{X} u(0)^{p-1}\dmu
    = \int_{B(x_0,r)} f^{p-1}\dmu,
\end{split}
\end{equation*}
where we applied \eqref{eq: mass conserv}. Hence \eqref{eq: application of PHI for Sobolev} implies 
\begin{equation}
\begin{split}
    \dashint_{B(x_0,r)}  u((4\sigma r)^p)^p\dmu 
    &\leq c \biggl(\frac{\mu(B(x_0,r)}{\mu(B(x_0,\sigma r))}\biggr)^\frac{p}{p-1} \biggl(\dashint_{B(x_0,r)} f^{p-1}\dmu\biggr)^{\frac{p}{p-1}}\\
    &\leq c \biggl( \frac{\mu(B(x_0,r)}{\mu(B(x_0,\sigma r))}\biggr)^\frac{p}{p-1}  \dashint_{B(x_0,r)} f^{p}\dmu.
\end{split}
\end{equation}
Since $\mu$ is doubling, \eqref{eq: doubling implies radii control from above} implies that there exists a constant $\alpha>0$, depending only on $c_H$, such that
\begin{equation*}
    \dashint_{B(x_0,r)}  u((4\sigma r)^p)^p\dmu\leq c \left( \frac{1}{\sigma}\right)^\frac{\alpha p}{p-1}  \dashint_{B(x_0,r)} f^{p}\dmu,
\end{equation*}
provided that $0<\sigma r <\frac13\diam(X)$. Hence, choosing $\sigma $ large enough we have
\begin{equation*}
    \dashint_{B(x_0,r)} u((4\sigma r)^p)^p\dmu\leq \frac{1}{2} \dashint_{B(x_0,r)} f^p\dmu,
\end{equation*}
and by \eqref{eq: basic estimate for sobolev} we conclude that
\begin{equation*}
    \dashint_{B(x_0,r)} f^p\dmu \leq cr^p \dashint_{B(x_0,r)} g_f^p\dmu + \frac{1}{2} \dashint_{B(x_0,r)} f^p\dmu.
\end{equation*}
Absorbing the second term on the right-hand side completes the proof.
\end{proof}

In the following, we will upgrade the Poincar\'e-type inequality for functions with zero boundary values in \Cref{prop: PHI implies local sobolev} to the Poincar\'e inequality in \Cref{def: poincare} by using the stationary solutions of the doubly nonlinear equation \eqref{eq: PDE}. This leads to the elliptic $p$-Laplace equation, formally written as
\begin{equation}
    \label{eq: p-Laplace}
    -\div (\lvert \nabla u\rvert^{p-2}\nabla u)=0.
\end{equation}
We consider variational solutions of \eqref{eq: p-Laplace} in the context of metric measure spaces.
\begin{definition}
Assume that $\Omega$ is an open set in $X$ and let $1<p<\infty$.
A function $u\in N^{1,p}(\Omega)$ is a solution of \eqref{eq: p-Laplace} in $\Omega$ if 
\begin{equation*}
    \int_{\Omega} g_u^p \dmu\leq \int_{\Omega} g_{u+\varphi}^p\dmu
\end{equation*}
for every $\varphi\in N^{1,p}_0(\Omega)$.
\end{definition}

We note that time invariant solutions of the doubly nonlinear equation \eqref{eq: PDE} are solutions to the $p$-Laplace equation
\eqref{eq: p-Laplace}.
Thus the parabolic Harnack inequality in Definition \ref{def: PHI} immediately implies the elliptic Harnack inequality below.

\begin{definition}
The elliptic Harnack inequality holds in $(X,d,\mu)$ if there exists a constant $c_H>0$ such that whenever $u$ is a nonnegative solution of \eqref{eq: p-Laplace} in $B(x_0,4r)$, then 
\begin{equation}
    \label{eq: elliptic harnack}
    \esssup_{B(x_0,r)} u \leq c_H\essinf_{B(x_0,r)} u.
\end{equation}
\end{definition}

The following lemma on the H\"older continuity of solutions to \eqref{eq: p-Laplace} is well known. We include a short proof to show that the result also holds in the setting of metric measure spaces.
The main advantage of considering the elliptic $p$-Laplace equation instead of the doubly nonlinear equation \eqref{eq: PDE} is that the class of solutions is invariant under the addition of constants, that is, adding a constant to a solution gives another solution of the same equation. This property is employed in the H\"older continuity result below. 

\begin{lemma}
\label{lemma: harnack implies Holder}
Assume that the elliptic Harnack inequality holds in $(X,d,\mu)$. Then there exist constants $c>0$ and $0<\kappa\leq 1$, depending only on $c_H$, such that
\begin{equation*}
    \essosc_{B(x_0,r)} u \leq c \left(\frac{r}{R}\right)^\kappa \essosc_{B(x_0,R)} u,
\end{equation*}
whenever $u$ is a solution of $\eqref{eq: p-Laplace}$ in $B(x_0,R)$.
\end{lemma}
\begin{proof}
Let $r_k = 4^{-k} R$, $k\in\mathbb N$. Applying \eqref{eq: elliptic harnack} to
\[
u-\essinf_{B(x_0,r_k)} u
\quad\text{and}\quad
\esssup_{B(x_0,r_k)} u - u,
\]
which both are nonnegative solutions of $\eqref{eq: p-Laplace}$ in $B(x_0,r_k)$, we obtain
\begin{equation}
    \label{eq: holder eq 1 from harnack}
    \esssup_{B(x_0,r_{k+1})} u - \essinf_{B(x_0,r_k)} u 
    \leq c \bigl(\essinf_{B(x_0,r_{k+1})} u - \essinf_{B(x_0,r_k)} u\bigr), 
\end{equation}
and 
\begin{equation}
    \label{eq: holder eq 2 from harnack}
    \esssup_{B(x_0,r_{k})} u - \essinf_{B(x_0,r_{k+1})} u 
    \leq c \bigl(\esssup_{B(x_0,r_{k})} u - \esssup_{B(x_0,r_{k+1})} u\bigr),
\end{equation}
for a constant $c>1$.
Summing up \eqref{eq: holder eq 1 from harnack} and \eqref{eq: holder eq 2 from harnack} gives
\begin{equation*}
    \essosc_{B(x_0,r_{k+1})} u + \essosc_{B(x_0,r_{k})} u 
    \leq c\bigl(\essosc_{B_(x_0,r_k)} u - \essosc_{B(x_0,r_{k+1})}u\big),
\end{equation*}
which implies 
\begin{equation}
    \label{eq: osc decay holder}
    \essosc_{B(x_0,r_{k+1})} u \leq \frac{c-1}{c+1} \essosc_{B(x_0,r_{k})} u.
\end{equation}
Let $\gamma=\frac{c-1}{c+1}$. 
Applying \eqref{eq: osc decay holder} recursively, we obtain
\begin{equation*}
    \essosc_{B(x_0,r_k)} u \leq \gamma^k \essosc_{B(x_0,R)} u. 
\end{equation*}
For $0<r\leq R$ we let $k$ be such that $r_k < r\leq r_{k-1}$, and let $\alpha=-\frac{\log \gamma}{\log 4}$. Then
\begin{equation*}
\begin{split}
    \essosc_{B(x_0,r)} u 
    &\leq \essosc_{B(x_0,r_k)} u \leq \gamma^{k-1} \essosc_{B(x_0,R)} u \\
    &= \gamma^{-1} \left( \frac{r_k}{R} \right)^\alpha \essosc_{B(x_0,R)} u \leq \gamma^{-1} \left( \frac{r}{R}\right)^\alpha \essosc_{B(x_0,R)} u,
\end{split}
\end{equation*}
which proves the claim.
\end{proof}

Next, we consider a more general form of the Poincar\'e-type inequality for zero boundary values which was proved in \Cref{prop: PHI implies local sobolev}, for a parameter $1<q<\infty$. We combine this inequality with the elliptic Harnack inequality and volume doubling to show an estimate which allows to localize arbitrary Sobolev functions.
\begin{lemma}
\label{Poincare with small ball on RHS}
Assume that $\mu$ is doubling and that the elliptic Harnack inequality holds in $(X,d,\mu)$. Let $1<q<\infty$ and assume that there exist constants $c>0$ and $\gamma>0$ such that
\begin{equation}\label{eq: local sobolev inequality}
    \biggl(\dashint_{B(x_0,r)} \lvert u\rvert^q \dmu\biggr)^\frac{1}{q} \leq cr \biggl(\dashint_{B(x_0,r)} g_u^p \dmu\biggr)^\frac{1}{p}
\end{equation}
for every $u\in N^{1,p}_0(B(x_0,r))$, where $x_0\in X$ and $0<r\leq\gamma\diam(X)$. Then there exists a constant $c$ depending on $p$, $q$, $c_H$, $c_\mu$ and the constant in \eqref{eq: local sobolev inequality}, such that for any $0<\delta<1$ we have
\begin{equation*}
    \biggl(\dashint_{B(x_0,r)}f^q\dmu\biggr)^\frac{1}{q} \leq c r \delta^{-Q} \biggl(\dashint_{B(x_0,4r)} g_f^p\dmu\biggr)^\frac{1}{p}+c \dashint_{B(x_0,\delta r)} f\dmu,
\end{equation*}
for any nonnegative $f\in N^{1,p}(B(x_0,4r))$, where $0<4r\leq \gamma \diam (X)$ and $Q$ is as in \eqref{eq: doubling implies radii control from below}.
\end{lemma}
\begin{proof}
We may assume that $f$ is bounded. Let $u$ be the solution of \eqref{eq: p-Laplace} in $B(x_0,4r)$ with $u-f\in N^{1,p}_0(B(x_0,4r))$. We note that
\begin{equation*}
    \biggl(\dashint_{B(x_0,r)}f^q\dmu\biggr)^\frac{1}{q} \leq \biggl( \dashint_{B(x_0,r)}\lvert f-u\rvert^q\dmu \biggr)^\frac{1}{q}+ \biggl(\dashint_{B(x_0,r)}u^q\dmu \biggr)^\frac{1}{q}.
\end{equation*}
Using the elliptic Harnack inequality, we estimate the second term on the right-hand side by
\begin{equation*}
\begin{split}
    \biggl(\dashint_{B(x_0,r)}u^q\dmu \biggr)^\frac{1}{q} 
    &\le\esssup_{B(x_0,r)}u
    \leq c \essinf_{B(x_0,r)}u 
     \leq c \essinf_{B(x_0,\delta r)}u \\
    &\leq c\dashint_{B(x_0,\delta r)} u \dmu 
    \leq c \dashint_{B(x_0,\delta r)} \bigl(\lvert u-f\rvert +f\bigr)\dmu.
\end{split}
\end{equation*}
By \eqref{eq: doubling implies radii control from below} we obtain 
\begin{equation}
\label{eq: localizing reverse holder before sobolev}
\begin{split}
    \biggl(\dashint_{B(x_0,r)}f^q\dmu\biggr)^\frac{1}{q} 
    & \leq c \biggl( \dashint_{B(x_0,r)}\lvert f-u\rvert^q\dmu\biggr)^\frac{1}{q}\\
    &\qquad+ c\dashint_{B(x_0,\delta r)}\lvert f-u\rvert\dmu 
    + c \dashint_{B(x_0,\delta r)}f\dmu  \\
    & \leq c \delta^{-Q} \biggl( \dashint_{B(x_0,4r)}\lvert f-u\rvert^q\dmu \biggr)^\frac{1}{q} + c\dashint_{B(x_0,\delta r)}f\dmu.
\end{split}
\end{equation}
Applying \eqref{eq: local sobolev inequality} to $f-u$ and using the minimizing property of $g_u$ gives
\begin{equation*}
\begin{split}
    &\biggl( \dashint_{B(x_0,4r)}\lvert f-u\rvert^q\dmu \biggr)^\frac{1}{q} 
     \leq c r \biggl(\dashint_{B(x_0,4r)}g_{f-u}^p\dmu\biggr)^\frac{1}{p} \\
    &\qquad \leq c r \biggl(\dashint_{B(x_0,4r)}(g_f^p +g_u^p)\dmu\biggr)^\frac{1}{p} 
    \leq c r \biggl(\dashint_{B(x_0,4r)}g_f^p\dmu\biggr)^\frac{1}{p}.
\end{split}
\end{equation*}
The claim follows by combining the estimate above with \eqref{eq: localizing reverse holder before sobolev}.
\end{proof}

In the following result we show that the full Poincar\'e inequality holds under the assumptions of  \Cref{Poincare with small ball on RHS}.
\begin{theorem}
\label{thm: sobo doubling ehi implies poincare}
Let $(X,d,\mu)$ be as in the statement of \Cref{Poincare with small ball on RHS}. Then a $(q,p)$-Poincar\'e inequality holds in $(X,d,\mu)$, with constants $c_P$ and $\tau$ depending only on $p$, $c_\mu$, $c_H$ and the constants in \eqref{eq: local sobolev inequality}. 
\end{theorem}

\begin{proof}
We will use the characterization in \Cref{lemma: charact of poinc} and show that \eqref{eq: charact of Poincare as reverse Holder} holds. Let $f\in N^{1,p}(B(x_0,4r))$ be a nonnegative function, where $0<4r\leq \gamma \diam (X)$. We may assume that $f$ is bounded. Let $u$ be the unique  solution of \eqref{eq: p-Laplace} with $u-f\in N^{1,p}_0(B(x_0,4r))$. Let $0<\delta<1$. Applying \Cref{Poincare with small ball on RHS} to $\lvert f-\essinf_{B(x_0,\delta r)} u\rvert$ we obtain
\begin{equation*}
\begin{split}
    &\biggl(\dashint_{B(x_0,r)} \lvert f-f_{B(x_0,r)}\rvert^q\dmu\biggr)^\frac{1}{q}  \leq 2\biggl(\dashint_{B(x_0,r)} \lvert f-\essinf_{B(x_0,\delta r)} u\rvert^q\dmu\biggr)^\frac{1}{q} \\
    &\qquad \leq c \delta^{-Q} r \biggl(\dashint_{B(x_0,4r)}g_f^p\dmu\biggr)^\frac{1}{p} + c \dashint_{B(x_0,\delta r)} \lvert f-\essinf_{B(x_0,\delta r)} u\rvert\dmu\\
    &\qquad\leq c \delta^{-Q} r \biggl( \dashint_{B(x_0,4r)}g_f^p\dmu\biggr)^\frac{1}{p} + c \dashint_{B(x_0,\delta r)} (u-\essinf_{B(x_0,\delta r)} u)\dmu \\
    &\qquad \qquad + c \dashint_{B(x_0,\delta r)} \lvert f-u\rvert\dmu \\
    &\qquad \leq c \delta^{-Q} r \biggl( \dashint_{B(x_0,4r)}g_f^p\dmu\biggr)^\frac{1}{p}+ c \essosc_{B(x_0,\delta r)} u.
\end{split}
\end{equation*}
In the last step we used \eqref{eq: doubling implies radii control from below} and \eqref{eq: local sobolev inequality}. By \Cref{lemma: harnack implies Holder} and the elliptic Harnack inequality we find that
\begin{equation*}
\begin{split}
    \essosc_{B(x_0,\delta r)} u 
    &\leq c \delta^{\kappa} \essosc_{B(x_0, r)} u
    \leq c \delta^{\kappa}\esssup_{B(x_0, r)} u 
    \leq c \delta^{\kappa} \dashint_{B(x_0,r)}u\dmu\\
    &\leq c \delta^{\kappa} \dashint_{B(x_0,r)}(\lvert f-u\rvert+ f)\dmu.
\end{split}
\end{equation*}
Applying again \eqref{eq: local sobolev inequality} to $f-u$ we have
\begin{equation*}
\begin{split}
    &\biggl(\dashint_{B(x_0,r)} \lvert f-f_{B(x_0,r)}\rvert^q\dmu\biggr)^\frac{1}{q}\\
    &\qquad\leq c \delta^{-Q} r \biggl(\dashint_{B(x_0,4r)}g_f^p\dmu \biggr)^\frac{1}{p}
    + c \delta^{\kappa}  \dashint_{B(x_0,r)}(\lvert u -f\rvert + f)\dmu \\ 
    &\qquad \leq c \delta^{-Q} r \biggl(\dashint_{B(x_0,4r)}g_f^p\dmu \biggr)^\frac{1}{p} + c \delta^{\kappa} \dashint_{B(x_0,r)} f\dmu.
\end{split}
\end{equation*}
Finally, we conclude that
\begin{equation*}
\begin{split}
    \biggl(\dashint_{B(x_0,r)} f^q\dmu\biggr)^\frac{1}{q} & \leq \biggl(\dashint_{B(x_0,r)} \lvert f-f_{B(x_0,r)}\rvert^q\dmu\biggr)^\frac{1}{q}+\dashint_{B(x_0,r)} f\dmu \\ 
    & \leq c \delta^{-Q} r \biggl(\dashint_{B(x_0,4r)}g_f^p\dmu\biggr)^\frac{1}{p} + (1+c \delta^{\kappa}) \dashint_{B(x_0,r)} f\dmu.
\end{split}
\end{equation*}
By choosing $\delta>0$ to be small enough we have shown that \eqref{eq: charact of Poincare as reverse Holder} holds.
It follows from \Cref{lemma: charact of poinc}  that a  $(q,p)$-Poincar\'e inequality holds in $(X,d,\mu)$.
\end{proof}
We are now ready to prove our main result.
\begin{proof}[Proof of \Cref{thm: main thm}]
The claim that $(1)$ implies $(2)$ is contained in \Cref{thm: PHI}. 
To prove that $(2)$ implies $(1)$, assume that the parabolic Harnack inequality holds. That $\mu$ is doubling follows from \Cref{thm: PHI implies doubling}. Combining \Cref{prop: PHI implies local sobolev} and \Cref{thm: sobo doubling ehi implies poincare} with $q=p$, we conclude that a $(p,p)$-Poincar\'e inequality holds in $(X,d,\mu)$, from which a $(1,p)$-Poincar\'e inequality follows immediately. 
\end{proof}


\end{document}